\documentclass[10pt]{amsart}
\usepackage{graphicx,amssymb,amsfonts,amsmath,amsthm,newlfont}
\usepackage{epsfig,url}
\usepackage{color}

\usepackage{axodraw4j}

\usepackage[all,2cell]{xy} \UseAllTwocells \SilentMatrices

\newtheorem{thm}{Theorem}[section]
\newtheorem{cor}[thm]{Corollary}
\newtheorem{lem}[thm]{Lemma}
\newtheorem{prop}[thm]{Proposition}
\theoremstyle{definition}
\newtheorem{defn}[thm]{Definition}

\newtheorem{ex}[thm]{Examples}
\newtheorem{example}[thm]{Example}

\theoremstyle{remark}
\newtheorem{rem}[thm]{Remark}

\numberwithin{equation}{section}
\newcommand{\Z}{\mathbb Z}
\newcommand{\C}{\mathbb C}

\newcommand{\R}{\mathbb R}
\newcommand{\N}{\mathbb N}

\newcommand{\Lef}{\mathbb{L} }

\newcommand{\GE}{\mathbb{G} }

\newcommand{\gr}{\mathrm{gr}}

\newcommand{\ZZ}{\mathcal{Z}}

\newcommand{\ssl}{\mathfrak{sl}}

\newcommand{\sfh}{\mathsf{h}}
\newcommand{\sfw}{\mathsf{w}}

\newcommand{\MI}{\mathcal{MI}}

\newcommand{\Q}{\mathbb Q}
\newcommand{\Li}{\mathrm{Li}}

\newcommand{\Be}{B}
\newcommand{\s}{\mathbf{s}}

\newcommand{\To}{\longrightarrow}

\newcommand{\G}{\mathbb{G}}

\newcommand{\tone}{\overset{\rightarrow}{1}\!}

\newcommand{\SL}{\mathrm{SL}}

\newcommand{\Or}{\mathcal{O}}

\newcommand{\mm}{\mathfrak{m} }

\newcommand{\HH}{\mathfrak{H} }

\newcommand{\id}{\mathrm{id} }

\newcommand{\LL}{\mathbb{L}}

\newcommand{\sv}{\mathrm{sv}}

\begin{document}
\author{Francis Brown}
\begin{title}[A class of non-holomorphic modular forms I]{A class of non-holomorphic modular forms I}\end{title}
\maketitle

This  paper  studies examples of     real analytic functions on the upper half plane
 satisfying a modular transformation property of the form  
\begin{equation} \label{firstintrobimod}  f \Big({ a z +b\over cz+d} \Big) =   (c z+d)^r (c\overline{z}+d)^s f(z)\ .
\end{equation}
for integers $r,s$.  They \emph{do not}  satisfy a simple condition involving the Laplacian. 
The \emph{raison d'\^etre} for this  class of functions is two-fold:

\begin{enumerate}
\vspace{0.05in}
\item Holomorphic modular forms  $f$ with rational Fourier coefficients correspond to certain pure motives $M_f$ over $\Q$. Using  iterated integrals, we can construct non-holomorphic modular forms which are associated to  iterated extensions of the pure motives $M_f$.  Their coefficients  are periods. 
\vspace{0.05in}
\item In genus one closed string perturbation theory, one assigns  a lattice sum to  a graph \cite{Graph2}, which  defines  a real-analytic function  on the upper half plane invariant under $\SL_2(\Z)$. It is an open problem to give
a complete description of this class of functions and prove their conjectured properties. 
\end{enumerate}

In this  introductory paper, we  describe elementary properties of a  
class $\mathcal{M}$ 
of  modular forms. Within this class are modular iterated integrals, which  are analogues of single-valued polylogarithms,  and are obtained by solving a differential equation in   $\mathcal{M}$.  The basic prototype are  real analytic Eisenstein series, defined by 
$$\mathcal{E}_{r,s}(z) = {w! \over ( 2 i \pi )^{w+1}} { 1\over 2 } \sum_{(m,n)\neq (0,0)}  { i \,\mathrm{Im}(z) \over (mz+n)^{r+1} (m \overline{z}+n)^{s+1}}$$
for all $r,s\geq 0$ such that  $w=r+s>0$ is even. It is known that the functions $\mathrm{Im}(z)^r \, \mathcal{E}_{r,r}(z)$ all occur as modular graph functions (2). 
Their relation with motives $(1)$ comes about by expressing the $\mathcal{E}_{r,s}$ as integrals. 
Indeed, they are equivariant or `modified single-valued versions' of  regularised Eichler integrals of  holomorphic Eisenstein series, and  their Fourier expansion involves the Riemann zeta values   $\zeta(w+1)$, which are periods of  simple extensions of Tate motives.  We shall say very little about motives in this paper,  and instead refer to  \cite{MMV, MEM} for geometric  motivation.
\\

This paper connects with the work of Don Zagier in several ways: through his work on modular graph functions \cite{Graph1}, on single-valued polylogarithms \cite{ZagierBWfunction}, on period polynomials  \cite{KZ},  on periods \cite{KohnenZagier}, on multiple zeta values \cite{GKZ}, on double Eisenstein series \cite{IKZ}, and doubtless many others.

 It is a great pleasure to dedicate this paper to him on  his 65th birthday.

\section{Modular graph functions}
For motivation, we briefly recall the   definition of modular graph functions.

\begin{defn} Let $G$ be a connected graph with no self-edges. It is permitted to have a number of half-edges.  Denote its set of  vertices by $V_G$ and number its edges  (including the half-edges) $1,\ldots, r$.  
Choose an orientation of $G$. The associated  modular graph function is defined, when it converges,  by the sum \cite{Graph1} (3.12):
$$I_{G}(z) =  \pi^{-r} \sum'_{m_1,n_1}    \ldots  \sum'_{m_r,n_r}   {\mathrm{Im}(z) \over |m_1 z+n_1|^2}\ldots  {\mathrm{Im}(z) \over |m_r z+n_r|^2} \prod_{v \in V_G}  \delta(m_v) \delta(n_v)\ $$
where $z$ is a variable in the upper half plane $\HH$, the prime over a summation symbol denotes a sum  over $(m,n) \in \Z^2 \backslash (0,0)$, and for every vertex $v \in V_G$
$$m_v = \sum^r_{i=1}    \varepsilon_{v,i} m_i \qquad \hbox{ and } \qquad n_v = \sum^r_{i=1}    \varepsilon_{v,i} n_i $$
where $\varepsilon_{v,i}$ is $0$ if the edge $i$ is not incident to the vertex $v$, $+1$ if $i$ is oriented towards the vertex $v$, and $-1$ if it is oriented away from $v$. 
\end{defn}

The function $I_G$ depends neither on the edge numbering, nor the choice of orientation of $G$.
It defines a function 
$I_G $ on the upper half plane
which is real-analytic and invariant under the action of $\SL_2(\Z)$. 

\begin{ex} Consider the  graph with 3 half-edges  depicted  on the left:
\begin{center}
\begin{figure}[h]
\fcolorbox{white}{white}{
  \begin{picture}(451,60) (90,-8)
    \SetWidth{1.0}
    \SetColor{Black}
%    \Vertex(194.539,-4.432){2}  %
    \Vertex(218.174,19.203){2}
%    \Vertex(218.174,54.655){2} %
 %   \Vertex(241.808,-4.432){2} %
    \SetColor{Black}
    \Line[arrow,arrowpos=0.5,arrowlength=2,arrowwidth=1,arrowinset=0.2](218.174,54.655)(218.174,19.203)
    \Line[arrow,arrowpos=0.5,arrowlength=2,arrowwidth=1,arrowinset=0.2](218.174,19.203)(194.539,-4.432)
    \Line[arrow,arrowpos=0.5,arrowlength=2,arrowwidth=1,arrowinset=0.2](218.174,19.203)(241.808,-4.432)
  %  \Vertex(312.713,54.655){2}
    \Vertex(312.713,19.203){2}
%    \Vertex(289.078,-4.432){2}
%    \Vertex(336.347,-4.432){2}
    \Vertex(312.713,36.929){2}
    \Line[arrow,arrowpos=0.5,arrowlength=2,arrowwidth=1,arrowinset=0.2](312.713,54.655)(312.713,36.929)
    \Line[arrow,arrowpos=0.5,arrowlength=2,arrowwidth=1,arrowinset=0.2](312.713,36.929)(312.713,19.203)
    \Line[arrow,arrowpos=0.5,arrowlength=2,arrowwidth=1,arrowinset=0.2](312.713,19.203)(289.078,-4.432)
    \Line[arrow,arrowpos=0.5,arrowlength=2,arrowwidth=1,arrowinset=0.2](312.713,19.203)(336.347,-4.432)
  \end{picture}
}
 \begin{caption}{Two graphs}\end{caption} \end{figure}
\end{center}
 The associated modular graph function is called
$$C_{1,1,1}(z) =  \pi^{-3} \sum'_{m_1,n_1,m_2,n_2}   { \mathrm{Im(z)}^3  \over |m_1 z+n_1|^2 |m_2z+n_2|^2 |(m_1+m_2)z+n_1+n_2|^2}$$
where the sum is over $(m_1,n_1) \in \Z^2$, $(m_2,n_2) \in \Z^2$ such that 
$$(m_1,n_1) \neq (0,0) \quad ,   \quad (m_2,n_2) \neq (0,0)  \quad , \quad (m_1+m_2,n_1+n_2) \neq (0,0) \ .$$
Zagier showed, in one of the  first calculations of a modular graph function,  that
$$C_{1,1,1}(z) =  {2\over 3}  \LL^2 \, \mathcal{E}_{2,2} +\zeta(3)$$
where $\LL = - 2 \pi \mathrm{Im}(z)$. 
See \cite{Graph1} Appendix B, for another derivation of this result. 
\end{ex}
 
\subsection{Properties}  \label{SectModGraphProperties} The literature on modular graph functions is too extensive to review  in detail here. 
 Instead, we give an incomplete list of the expected and conjectural properties of these functions and refer to \cite{Graph1}, \cite{Graph2}, \cite{Graph3},  \cite{Graph4}, \cite{Zerbini} for further details.
   \begin{enumerate}
   \vspace{0.05in}
  \item  Zerbini  \cite{Zerbini} has shown that  in all known examples,   the zeroth modes of modular graph functions involve 
  a certain class of multiple zeta values
  $$\zeta(n_1,\ldots, n_r) = \sum_{1\leq k_1 < \ldots < k_r} {1 \over n_1^{k_1} \ldots n_r^{k_r}}$$
  where $n_1,\ldots, n_r \in \N$ and $n_r \geq 2$, which are called  `single-valued' multiple zeta values.  The quantity $r$ is called the depth. 
  The `single-valued' subclass is generated in depth one by odd zeta values $\zeta(2n+1)$ for  $n\geq 1$, in depth two by 
  products $\zeta(2m+1)\zeta(2n+1)$, but starting from depth three  includes the following combination  of triple zeta values 
  $$\zeta_{\sv} (3,5, 3):  =   2 \zeta(3,5,3) - 2 \zeta(3) \zeta(3,5) -10   \zeta(3)^2 \zeta(5)\ . $$ 

\vspace{0.05in}
 \item  The $I_{G}$ satisfy some mysterious inhomogeneous  Laplace eigenvalue equations. A simple example of this is the equation   \cite{Graph4} (1.4)  
 \begin{equation} \label{introC211evalue} 
 (\Delta +2)\, C_{2,1,1}(z) =  16 \, \LL^2 \, \mathcal{E}_{1,1}^2     - \textstyle{2\over 5} \,\LL^3 \, \mathcal{E}_{3,3}
 \end{equation}
 where $\Delta$ is the Laplace-Beltrami operator. The function $C_{2,1,1}$ corresponds to the modular graph function of the graph with four edges and five vertices depicted above on the right.   Note that the operator $\Delta$ in the physics literature has the opposite sign from the usual convention $(\ref{Delta00Laplacedef})$. 
   
  \vspace{0.05in}
\item  Modular graph functions satisfy many relations \cite{Graph3}, which suggests that they should lie in  a finite-dimensional space of modular-invariant functions.

\vspace{0.05in} 
\item The zeroth modes of modular graph functions are homogeneous \cite{Graph1}, \S6.1, for a  grading called the weight, in which rational numbers have weight $0$, and multiple zeta values have weight $n_1+\ldots +n_r$. The weight of $\mathrm{Im}(z)$ is zero. 
 \end{enumerate}

\noindent
In the continuation  of this paper, we  construct a class of functions $\MI^E \subset \mathcal{M}$ satisfying  $(1)$-$(5)$ (see \S\ref{sectMIE}).  They
 are associated to  universal mixed elliptic motives \cite{MEM}, which are in turn  related  to mixed Tate motives over the integers.

\subsection{Landscape}
A heuristic explanation for the connection between string theory and our modular iterated integrals   can be summarised in the following picture:

\begin{table}[h]
%\caption{default}
\begin{center}
\begin{tabular}{|c|c|c|}
\hline
&\hbox{Open string} & \hbox{Closed string} \\ \hline
\hbox{Genus 0} & Multiple polylogs  &  Single-valued polylogs\\
\hbox{Genus 1} & Multiple elliptic polylogs  &   Equivariant iterated Eisenstein integrals \\ \hline
\end{tabular}
\end{center}
\label{default}
\end{table}%

\noindent

The open genus zero amplitudes are integrals on the moduli spaces of curves of genus $0$ with $n$ marked points $\mathfrak{M}_{0,n}$. They   involve multiple polylogarithms, whose values are multiple zeta values. 
The genus one string amplitudes are integrals on the moduli space $\mathfrak{M}_{1,n}$ and  are expressible \cite{Elliptic1loop} in terms of multiple elliptic polylogarithms \cite{MEP}. Viewed as a function of the modular parameter, the latter are given by  
certain iterated integrals of Eisenstein series.  The passage from the open to the closed string involves a  `single-valued' construction \cite{SteibergerSV}.  The closed superstring amplitudes in genus one are thus  linear combinations of iterated integrals of Eisenstein series and their complex conjugates which are modular.  This is the definition of the space $\MI^E$. A rigorous proof of the relation between closed superstring amplitudes and our class $\MI^E$ might go along the broad lines of the author's thesis,  generalised to genus one using \cite{MEP}.

\subsection{Acknowledgements} Many thanks to 
Michael Green, Eric d'Hoker, Pierre Vanhove, Don Zagier and  Federico Zerbini for explaining 
properties of modular graph functions to me. Many thanks also to Martin Raum for pointing out  connections with the literature on mock modular forms.  This work was partially supported by ERC grant 724638.
  I strived to make the exposition in this paper as accessible and elementary as possible.
As a result, there is considerable overlap with some classical constructions and well-known results in the theory of modular forms.  I apologise in advance if I have failed to provide attributions in every case.

\section{A class of functions  $\mathcal{M}$}
Throughout this paper, $z$ will denote a variable in the upper half plane
$$\HH = \{ z: \mathrm{Im} \, z >0\} $$
 equipped with the standard action of $\SL_2(\Z)$:
\begin{equation}\label{gammaact} \gamma (z) = {az + b \over cz +d} \quad \hbox{ where } \quad \gamma = \begin{pmatrix} a & b\\c & d 
  \end{pmatrix}  \in \SL_2(\Z)\ .
  \end{equation} 
We shall  write  $z= x+iy$, and $q= \exp(2 i \pi z)$.  Let 
 \begin{equation} \label{LLdef} \LL = \log |q| = {1 \over 2 } \log q \overline{q} =  i \pi (z- \overline{z}) = - 2 \pi y   \ .
\end{equation} 

\subsection{First definitions}

\begin{defn} Call a real analytic function 
 $f: \HH \rightarrow \C$   \emph{modular of weights} $(r,s)$ if for all $\gamma \in \SL_2(\Z)$ of the form $(\ref{gammaact})$ it satisfies 
\begin{equation} \label{fbimod}  f(\gamma(z)) =   (c z+d)^r (c\overline{z}+d)^s f(z)\ .
\end{equation}
If $r+s$ is odd then $f$ vanishes  (put $\gamma = - \id$ in $(\ref{fbimod})$). Let $\mathcal{M}_{r,s}$
denote the space of real analytic functions of modular weights $(r,s)$ which admit an expansion of the form
\begin{equation} \label{introfexpansion}  f (q) \quad \in \quad  \C [[q,\overline{q}]][\LL^{\pm} ]\end{equation}
\end{defn}

A more general class of functions was considered in \cite{Pasles}. 
A function in $\mathcal{M}_{r,s}$ can be written explicitly, for some $N\in \N$, in the form  
\begin{equation}\label{qexp}  f  = \sum_{k=-N}^{N} \sum_{m,n\geq 0} a^{(k)}_{m,n}   \LL^k  q^m \overline{q}^n  
\end{equation}
where $a^{(k)}_{m,n}\in \C$. 
For any ring $R\subset \C$, let $\mathcal{M}(R)$ be the bigraded subspace of modular forms  whose  coefficients $a^{(k)}_{m,n}$ lie in  $R$.   Define a bigraded vector space 
 $$\mathcal{M} = \bigoplus_{r,s} \mathcal{M}_{r,s}$$
which is  a bigraded algebra since 
$\mathcal{M}_{r,s} \mathcal{M}_{k, l} \subset \mathcal{M}_{r+k, s+l}$.
Complex conjugation induces an involution   
$$ f(z) \mapsto \overline{f(z)} \quad  : \quad    \mathcal{M}_{r,s}  \overset{\sim}{\To}   \mathcal{M}_{s,r} $$
 which fixes $\LL \in \mathcal{M}_{-1,-1}$. 
Of special importance are  the  quantities
 \begin{equation}
 w=r+s \qquad \hbox{ and} \qquad h = r-s \ .
 \end{equation} 
We call $w$  the total weight. 
 We only consider the cases where $w, h$ are even.

\subsection{$q$  - expansions and pole filtration}

\begin{lem} \label{lemlogqexpansion}
Suppose that  $f: \HH \rightarrow \C$ satisfies equation $(\ref{fbimod})$,  and admits an expansion in the ring 
$ \C [[q, \overline{q}]]  [  \log q  ,  \log \overline{q}  ]$. 
Then $f \in \C[[q, \overline{q}]] [  \LL ]. $
\end{lem} 
\begin{proof} 
 Setting $\gamma =T$ in equation $(\ref{fbimod})$ gives  $f(z+1)= f(z)$. Since $q$ and $\overline{q}$ are invariant   under translations $z\mapsto z+1$, it suffices to show that 
$$\C [\log q, \log \overline{q}]^T = \C [\log |q|]$$
where $T$ denotes analytic continuation of $q$ around a loop around $0$ in the punctured $q$-disc.  We have 
$T \log q = \log q + 2 i \pi$ and $T \log \overline{q} = \log \overline{q} - 2 i \pi$. It is a  simple exercise in invariant theory
 to show that every $T$-invariant polynomial in $\log q$ and $\log \overline{q}$ is a  polynomial in $2 \log|q| = \log q + \log \overline{q}$. 
\end{proof} 

Every element $f\in \mathcal{M}$ admits  a $q$-expansion of the form  $(\ref{qexp})$ for some $N$.  
 This expansion is unique.  Define the \emph{constant part}  of $f$ to be 
$$f^0   =  \sum_k  a^{(k)}_{0,0}   \LL^k   \qquad \in \qquad  \C [\LL^{\pm}] \ .$$
The reason for calling this `constant', although it is not constant as a function on $\HH$, is that it is constant with respect to differential operators
to be defined below. 
In the physics literature, the constant parts of modular graph functions are called their `zeroth Fourier modes'.  
The space $\mathcal{M}$ is filtered by the order of poles in $\LL$. Set  
\begin{equation} \label{Polefilt} P^p \mathcal{M} = \{ f \in \mathcal{M}:    a^{(k)}_{m,n}(f) = 0  \quad  \hbox{ if } \quad  k<p\} \ .
\end{equation} 
It is a decreasing filtration. It satisfies $P^a \mathcal{M} \times P^b \mathcal{M} \subset P^{a+b} \mathcal{M}$, and $P^0 \mathcal{M}$ is the subalgebra of
functions admitting expansions in $\C[[q, \overline{q}]][\LL] $ with no poles in $\LL$. 
Multiplication by $\LL$ is an isomorphism 
$\Lef :    P^a \mathcal{M}_{i,j} \overset{\sim}{\rightarrow} P^{a+1} \mathcal{M}_{i-1,j-1} \ .$

\begin{example}  Consider the Eisenstein series,    defined for all even $k\geq 4$ by 
\begin{equation} \label{introEisdefn} \GE_{k}(q) = - {b_{k} \over 2k} + \sum_{n \geq 1} \sigma_{k-1}(n) q^n \qquad \in \mathcal{M}_{k,0}(\Q) \ , 
\end{equation}
where $\sigma$ denotes the divisor function.  
The Eisenstein series of weight two
$$\GE_2(q) = {-1 \over 24} + \sum_{n=1}^{\infty} \sigma_1(n) q^n = -{1 \over 24} + q + 3q^2 +4 q^3+7 q^4 +6 q^5+ \ldots $$
is not modular invariant, but can be modified \cite{Zag123} \S2.3 by defining
\begin{equation} \label{E2stardef}
\GE^*_2 = \GE_2 - {1 \over 4 \LL}\ 
\end{equation} 
which is modular of weight $2$ and therefore  defines an element in $\mathcal{M}_{2,0}$.  
Then, for example, the function $\LL^2 \GE_2^* \overline{\GE_2^*} \in \mathcal{M}_{0,0}$ is modular invariant, where $\overline{\GE_2^*} = \overline{\GE_2} - {1 \over 4\LL}$. 
\end{example} 

Recall that the polynomial ring
\begin{equation} \widetilde{M}:=   M[\GE_2^*]  
\end{equation}
where $M$ is the ring of holomorphic modular forms, 
is called the ring of almost holomorphic modular forms. By the previous example,  it is contained in $\mathcal{M}$.

\subsection{Differential  operators  (Maass)}

\begin{defn} \label{partialdefn} For any integers $r,s\in \Z, $ define a pair of operators
\begin{equation} \partial_r = (z - \overline{z}) {\partial \over \partial z} + r  \quad  , \quad \overline{\partial}_s = (\overline{z} - z) {\partial \over \partial \overline{z}} + s \ .
\end{equation}
They act on real analytic functions $f: \HH\rightarrow \C$.
\end{defn} 
These operators satisfy a version of the Leibniz rule:
\begin{equation} \label{partialLeibniz} \partial_{r+s} (fg) = \partial_r(f) g + f \partial_s(g) 
\end{equation} 
for any $r,s$ and  $f, g: \HH \rightarrow \C$,  and in addition the formula
\begin{equation}  \label{partialLk} \partial_r \big( (z-\overline{z})^k  f  \big)= (z-\overline{z})^k \partial_{r+k} f
\end{equation} 
for any integers $r,k$.   Both formulae  $(\ref{partialLeibniz})$ and $(\ref{partialLk})$ remain true on replacing  $\partial$ by $\overline{\partial}$ and are verified by straightforward computation.  Finally, one checks that 
\begin{equation} \label{partialcommutators}
\partial_{r-1} \big( \overline{ \partial}_s  f \big) -  \overline{\partial}_{s-1}  \big( \partial_r f \big) = (r-s) f \ .
\end{equation} 
The following lemma implies that these operators respect modular transformations.
\begin{lem} \label{lemdeltaauto} For all $\gamma \in M_{2\times 2}(\R)$  of the form $(\ref{gammaact})$, and $z \in \HH$,  we have 
\begin{eqnarray} \partial_r   \Big(  (c z+ d)^{-r} f ( \gamma z ) \Big)  &=  & (cz+d)^{-r-1} (c\overline{z}+d) \, (\partial_{r} f ) ( \gamma z )  \nonumber \\ 
\overline{\partial}_s   \Big(  (c \overline{z}+ d)^{-s} f  (\gamma z) \Big)  &=  & (cz+d)  (c\overline{z}+d)^{-s-1} (\partial_{s} f) ( \gamma z) \nonumber  \ .
\end{eqnarray}
\end{lem}

\begin{proof} Direct computation. 
\end{proof} 
See \S\ref{sectEquivariant} for another interpretation of $\partial_r, \overline{\partial}_s$  in terms of sections of vector bundles.
\begin{lem} \label{lempartialsqexp}
The operators $\partial_r$, $\overline{\partial}_s$ preserve the expansions $(\ref{qexp})$, the filtration $(\ref{Polefilt})$,  and are defined over $\Q$. Their action is given explictly for any $k,m,n$ by 
\begin{eqnarray}  \label{partialactformula} \partial_r (\LL^k\,  q^m \overline{q}^n )  &=  & (2m\,  \LL + r + k )  \, \LL^k  q^m \overline{q}^n  \\ 
\overline{\partial}_s (\LL^k\,  q^m \overline{q}^n )  &=  & (2n \, \LL + s + k ) \, \LL^k  q^m \overline{q}^n \ .\nonumber 
\end{eqnarray}
\end{lem} 
\begin{proof} The first part follows immediately from the formulae $(\ref{partialactformula})$, which are easily derived from the definitions.  The second line follows by complex conjugation. 
\end{proof} 

\begin{cor}
The operators $\partial_p, \partial_q$ preserve modularity: 
$$\partial_p : \mathcal{M}_{p,q} \To \mathcal{M}_{p+1, q-1} \qquad \hbox{ and } \qquad  \overline{\partial}_q : \mathcal{M}_{p,q} \To \mathcal{M}_{p-1, q+1} $$
\end{cor} 
\begin{proof} This  follows immediately from lemmas \ref{lemdeltaauto} and  \ref{lempartialsqexp}.
\end{proof} 
\begin{defn} Let us define linear operators 
$$\partial , \overline{\partial}:  \mathcal{M} \To \mathcal{M}$$
of bi-degrees $(1,-1)$ and $(-1,1)$ respectively, where $\partial$ acts on the component $\mathcal{M}_{r,s}$ via $\partial_r$ for all $s$, 
and similarly, $\overline{\partial}$ acts on $\mathcal{M}_{r,s}$ via $\overline{\partial}_s$ for any $r$. 
\end{defn} 
The operator  $\partial$ is a  derivation, i.e, 
$\partial (fg)  = \partial(f) g + f \partial (g) $
for all $ f, g \in \mathcal{M}$, and similarly for $\overline{\partial}$.  This follows, component by component, from the formula $(\ref{partialLeibniz})$. Likewise, it  commutes  with multiplication by $\LL^k$: 
$$\partial (\LL^k\,  f) = \LL^k \,  \partial (f)$$
for all $k$ and all $f\in \mathcal{M}$, and similarly for $\overline{\partial}$.  This is equivalent to  $(\ref{partialLk})$.
We can rewrite the previous equation in the form 
$$[\partial, \LL] =[\overline{\partial}, \LL] = 0\ ,$$
or think of $\LL$ as being  constant:  $\partial (\LL) = \overline{\partial} (\LL)=0$. 
\subsection{Action of $\ssl_2$} 
The equation $(\ref{partialcommutators})$ implies that 
$$  [\partial,  \overline{\partial} ] = \sfh $$
where we define the linear map 
\begin{equation} \label{hmapdefn} \sfh : \mathcal{M} \To \mathcal{M}
\end{equation}  to be multiplication by $r-s$ on the component $\mathcal{M}_{r-s}$. 
\begin{prop} \label{propsl2} The operators $\partial, \overline{\partial}$ generate a copy of the Lie algebra $\ssl_2$:
\begin{equation}
[\sfh ,\partial  ] =   2 \partial \qquad \ , \qquad  [ \sfh , \overline{\partial}  ] = -   2 \overline{\partial} \qquad \ , \qquad [\partial, \overline{\partial}]=\sfh\  
\end{equation} 
acting upon $\mathcal{M}$. Every element commutes with multiplication by  $\LL^k$. 
\end{prop}  
\begin{proof} Straightforward computation.
\end{proof} 

\subsection{Almost holomorphic modular forms}  \label{remalmosthol} The  subspace $\widetilde{M}[\LL^{\pm}]$ of almost holomorphic modular forms inherits an $\ssl_2$ module structure which is not to be confused with another $\ssl_2$ module structure  \cite{Zag123} \S5.3, which involves multiplication by $\GE_2$. 
For the convenience of the reader, we describe the differential module structure here.

 Let us define a new generator
$$ \mm :=  4  \LL \GE_2^*   =  4 \LL \GE_2 - 1   \qquad \in \quad \mathcal{M}_{1,-1} \ .$$
Then the ring 
$M [ \LL, \mm]$
is an $\ssl_2$-module with the following structure:  $\overline{\partial}(\LL) =0$, and 
\begin{equation} \label{partialbaronalmosthol}  \overline{\partial}\, \mm = 1 \qquad ,   \qquad \overline{\partial}{f} =0 \qquad \hbox{ for all } f \in M \ . 
\end{equation} 
Therefore $\overline{\partial}\big|_{M[\LL,\mm]} = \textstyle{ \partial \over \partial \mm}$ is simply differentiation with respect to $\mm$.  On the other hand, by looking at their first few Fourier coefficients, we easily verify that:
\begin{eqnarray}
\partial{\mm}  &=  &  -  \mm^2  + \textstyle{20 \over 3} \LL^2 \GE_4 \nonumber \\ 
\partial{\GE_4}  &= &-4 \mm \GE_4 + \textstyle{7\over 5} \LL \GE_6 \nonumber \\ 
\partial{\GE_6} & = & -6 \mm \GE_6 + \textstyle{800\over 7} \LL \GE_4^2 \nonumber 
\end{eqnarray} 

Since the ring of holomorphic modular forms  $M$ is generated by $\GE_4$ and $\GE_6$, 
we conclude that  $M[ \LL, \mm]$ is indeed closed under the action of $\partial$. These formulae are equivalent to a computation due to Ramanujan.
In general, for any $f\in M_{n}$ we have 
\begin{equation} \label{Serrederivative}  \partial f =   - n  f\,  \mm + 2 \, \theta(f) \LL \end{equation} 
where $\theta(f) \in M_{n+2}$ is  the `Serre derivative' of $f$ \cite{Zag123} $(53)$.  The previous formula is compatible with the commutation relation $h = [\partial, \overline{\partial}]$, as the reader may wish to check.

For example, the Hecke normalised cusp form  $\Delta$ of weight $12$ satisfies 
$\theta(\Delta) = 0$. It follows that 
$ \partial (\Delta)  = -12 \mm \Delta$,
which gives another interpretation of $\mm$.

\subsection{Bigraded Laplace operator}
By taking polynomials in  $\LL, \partial$ and $\overline{\partial}$ one can define any number of operators acting on the space $\mathcal{M}$. Examples include the Laplace operator,  Rankin-Cohen brackets  \S\ref{sectRC},
and the Bol operator (see \cite{Mock}). 

\begin{defn}For all integers $r,s$, 
define  a Laplace operator 
\begin{eqnarray}  \label{LaplaceDef}
 \Delta_{r,s} &=  &  - \overline{\partial}_{s-1} \partial_r   + r(s-1)   \\ 
  & = &   -  \partial_{r-1} \overline{\partial}_{s} +s(r-1)\ .\nonumber  
  \end{eqnarray} 
  The second definition is equivalent to the first by the commutation relation $(\ref{partialcommutators})$.
  These operators  are compatible with  complex conjugation: 
$\overline{\Delta_{r,s} f} = \Delta_{s,r} \overline{f}$.
\end{defn} 
From the definition and the formula $z=x+iy$,  one verifies that 
\begin{eqnarray} \Delta_{r,s}  &=  & - 4  y^2 {\partial \over \partial z}  {\partial \over \partial {\overline{z}}} + 2 i r  y \, { \partial \over \partial \overline{z}} - 2 i s  y \,{\partial \over \partial z}  \nonumber \\
& = &  \Delta_{0,0} + i (r-s) y {\partial \over \partial x} - (r+s) y {\partial \over \partial y} \ ,\nonumber 
\end{eqnarray} 
where  $\Delta_{0,0}$ is the usual hyperbolic Laplacian  
\begin{equation} \label{Delta00Laplacedef} \Delta_{0,0} = - y^2 \Big( {\partial^2 \over \partial x^2 } +{\partial^2\over \partial  y^2} \Big)\ ,
\end{equation}
 and  $\Delta_{r,-r}$ is the weighted hyperbolic Laplacian in the theory of Maass waveforms \cite{Maass}. 
 It follows from the previous computation $(\ref{partialactformula})$ that $\Delta_{r,s}$ acts via:
\begin{multline}  \label{Deltars} \Delta_{r,s} (   \LL^k q^m \overline{q}^n) =  \\ 
\Big( - 4mn \,  \LL^2 +2 ( kn+km+rn+sm)\,\LL -k(k+r+s-1) \Big)  \LL^k q^m \overline{q}^n\ .  \end{multline}
which has integral coefficients.  The modular transformation  properties of lemma \ref{lemdeltaauto} imply that the Laplace operator preserves the transformation law $(\ref{fbimod})$.

\begin{cor} The operator $\Delta_{r,s}$ defines a linear map
$$\Delta_{r,s}: \mathcal{M}_{r,s} \To \mathcal{M}_{r,s}\ .$$
In particular, the hyperbolic Laplacian $\Delta_{0,0}$ acts on the modular-invariant space $\mathcal{M}_{0,0}$. 
\end{cor} 

 Let 
$\Delta : \mathcal{M} \rightarrow \mathcal{M}$ denote
 the linear operator which acts by $\Delta_{r,s}$ on  $\mathcal{M}_{r,s}$. 
 Let  $\sfw: \mathcal{M} \rightarrow \mathcal{M}$ 
be the linear map which acts by multiplication by  $w=r+s$ on $\mathcal{M}_{r,s}$.

  \begin{lem} \label{lemoperatoridentities} The Laplace  operator  satisfies  the equations
   \begin{equation} \label{DeltaLL} \, (\Delta +\sfw ) \LL f = \LL \, \Delta f \end{equation}
i.e., $[\Lef, \Delta ] = \sfw  \Lef$, and also
$  [ \partial,  \Delta  ] = [   \overline{\partial}, \Delta] = 0$.
\end{lem}
\begin{proof}   By  (\ref{LaplaceDef}),  for any $f$ we have
\begin{multline} \Lef  (\Delta_{r,s} f )= \Lef ( - \overline{\partial} \partial  f + r(s-1) f)  =  ( - \overline{\partial} \partial   + r(s-1)) \Lef f   = \nonumber \\
  ( - \overline{\partial} \partial   + (r-1)(s-2)) \Lef f + (r+s-2) \Lef f =  \Delta_{r-1,s-1} \Lef f + (r+s-2) \Lef f\ ,
 \end{multline}
 which implies (\ref{DeltaLL}). Similarly, 
$$\partial  (\nabla_{r,s} f) = \partial ( - \overline{\partial} \partial f + r(s-1) f) = (- \partial  \overline{\partial} + (r+1-1)(s-1) ) \partial f = \nabla_{r+1,s-1} (\partial f)$$
which implies that $[\partial ,\nabla]=0$. By complex conjugating, $[\overline{\partial} ,\nabla]=0$. 
 \end{proof}

\subsection{Real analytic Petersson inner product}
Let 
$$\mathcal{D} = \{ |z| >1 , |\mathrm{Re} (z) | <  \textstyle{1\over 2} \} \qquad \hbox{ and } \qquad d\hbox{vol} = \displaystyle{dx dy \over y^2}  $$
be the interior of the standard fundamental domain for the action of $\SL_2(\Z)$ on $\HH$, and  the $\SL_2(\Z)$-invariant volume form on $\HH$ in its standard normalisation. 
For any $r,s$ let 
$$\mathcal{S}_{r,s} \subset \mathcal{M}_{r,s}$$
denote the subspace of functions $f$ whose constant part $f^0$ vanishes.  If $\mathcal{S}= \bigoplus_{r,s} \mathcal{S}_{r,s}$, there is an exact sequence
$$0 \To \mathcal{S}  \To \mathcal{M} \To \C[\LL^{\pm}] \To 0 $$
where the third map is the `constant part' $ f\mapsto f^0$. 
\begin{defn}  For any integer $n$
consider the  pairing 
\begin{eqnarray}  \label{IPpairing}
\mathcal{M}_{r,s} \times \mathcal{S}_{n-s,n-r}  & \To & \C \\
f \quad \times \quad  g  \qquad & \mapsto & \langle f, g\rangle :=  \int_{\mathcal{D}} f (z)   \overline{g(z)} \,  y^n  \, d\hbox{vol}  \nonumber 
\end{eqnarray} 
The function $f (z)\overline{g(z)} y^n$ is modular of weights $(0,0)$ and lies in $\mathcal{S}_{0,0}$.  The pairing $(\ref{IPpairing})$ coindices with the usual Petersson inner product when restricted to holomorphic modular forms. To verify that the integral is finite, it suffices to bound the integrand near  the cusp. Any element of $\mathcal{M}$ grows at most polynomially in $y$ as  $y\rightarrow \infty$, but via  $ q = \exp(2  \pi  i x) \exp(-2 \pi y)$ and $\overline{q} = \exp(-2 \pi i x) \exp(-2 \pi y)$ it tends to zero in absolute value exponentially fast  in $y$ at the cusp  since $g(z)\in \mathcal{S}$.  
\end{defn}

Two spaces $\mathcal{M}_{r,s}$ and $\mathcal{S}_{r',s'}$ can be paired via $(\ref{IPpairing})$ if and only if $r-s= r'-s'$. Equivalently,  $\langle f,g\rangle$ exists whenever  $h(f) = h(g)$, where $h$ was defined in  $(\ref{hmapdefn})$.

The pairing $(\ref{IPpairing})$ satisfies 
$$\langle \overline{f}, \overline{g} \rangle =  \overline{ \langle f, g \rangle } \ , $$
and, for any $e, f \in \mathcal{M}$ and $g\in \mathcal{S}$ such that $h(e)+h(f) = h(g)$,  we have 
$ \langle  f e , g \rangle = \langle  f , \overline{e} g\rangle$.  Via $(\ref{LLdef})$, we also have   for all $m \in \Z:$ 
\begin{equation} \label{scaleIP}
\langle f , \LL^m g \rangle =  \langle  \LL^m f , g \rangle = (- 2 \pi)^m \,  \langle f, g \rangle  \ .
\end{equation} 

We now consider special cases of this pairing. 
 When $n = r+s$,  we have
$$\langle  \ , \  \rangle: \mathcal{M}_{r,s} \times \mathcal{S}_{r,s} \To \C\ $$
 which restricts to a positive definite quadratic form on $\mathcal{S}_{r,s}$, since
 $$ \langle f, f \rangle  =  \int_{\mathcal{D}}   \big| f (z)    \big|^2   \,  y^n  \, d\hbox{vol} \quad   > \quad 0 \qquad \hbox{ for } f \in \mathcal{S}_{r,s} \ .  $$

 \subsection{Holomorphic projections \cite{Sturm}.} In the particular case $n= r$, we have  
 \begin{eqnarray} \mathcal{M}_{r,s} \times S_{r-s}  &\To &  \C \ .  \\
                  f \quad \times \quad  g  \quad & \mapsto &  \langle f, g \rangle  \nonumber
                     \end{eqnarray}
 where $S_{r-s} \subset \mathcal{S}_{r-s,0}$ is the space of holomorphic cusp forms of weight $r-s$. It is non-trivial only if $h(f) = r-s\geq12$.
 Similarly,  by setting $n=s$ we obtain
 \begin{eqnarray} \mathcal{M}_{r,s} \times \overline{S}_{s-r}  &\To &  \C \ .  \\
                  f \quad \times \quad  g  \quad & \mapsto &  \langle f,  \overline{g} \rangle  \nonumber
                     \end{eqnarray}
  which is non-trivial whenever $h(f)= r-s \leq - 12$.

  Equivalently, these two maps can be combined into  a single linear map
 $$ \mathcal{M}_{r,s} \To \mathrm{Hom}(S_{r-s},\C) \oplus   \mathrm{Hom}(\overline{S}_{s-r},\C)  \ , $$
 at least one component of which is zero. 
 Since the classical Petersson inner product restricts to a  non-degenerate quadratic form on  $S_{r-s}$,  we can identify  $\mathrm{Hom}(S_{r-s},\C)$  with $S_{r-s}$, and similarly for its complex conjugate. Via this identification, the previous map defines a  projection
 \begin{equation} \label{holprojection} p = (p^{h} , p^{a}) : \mathcal{M}_{r,s} \To  S_{r-s} \oplus \overline{S}_{s-r} \  , \end{equation}
 whose components we call the holomorphic and anti-holomorphic projections.  By taking the direct sum over $r$ and $s$, this  defines a linear map 
 \begin{equation}    p = (p^{h} , p^{a}) : \mathcal{M}  \To  S  \oplus \overline{S}\ . \end{equation}
 
\subsection{A picture of $\mathcal{M}$} The bigraded algebra $\mathcal{M}$ can be depicted as follows.

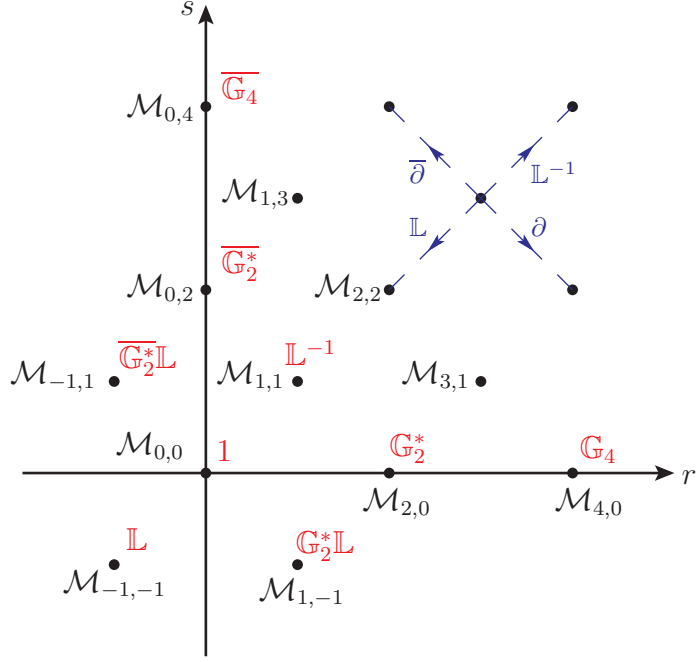
\begin{center}
\fcolorbox{white}{white}{
  \begin{picture}(251,258) (33,-18)
    \SetWidth{1.0}
    \SetColor{Black}
    \Line[arrow,arrowpos=0,arrowlength=5,arrowwidth=2,arrowinset=0.2,flip](103.642,223.838)(103.642,-17.993)
    \Line[arrow,arrowpos=1,arrowlength=5,arrowwidth=2,arrowinset=0.2](34.547,51.101)(276.379,51.101)
    \Vertex(103.642,51.101){2.036}
    \Vertex(138.19,85.649){2.036}
    \Vertex(103.642,120.196){2.036}
    \Vertex(103.642,189.291){2.036}
    \Vertex(138.19,154.744){2.036}
    \Vertex(172.737,51.101){2.036}
    \Vertex(172.737,120.196){2.036}
    \Vertex(207.284,85.649){2.036}
    \Vertex(241.832,51.101){2.036}
    \Vertex(69.095,85.649){2.036}
    \Vertex(69.095,16.554){2.036}
    \Vertex(138.19,16.554){2.036}
    \Vertex(207.284,154.744){2.036}
    \Vertex(241.832,120.196){2.036}
    \Vertex(172.737,189.291){2.036}
    \Vertex(241.832,189.291){2.036}
    \Text(160,34)[lb]{ \Large{\Black{$\mathcal{M}_{2,0}$}}}
    \Text(236.074,34)[lb]{\Large{\Black{$\mathcal{M}_{4,0}$}}}
    \Text(284,47)[lb]{\Large{\Black{$r$}}}
    \Text(94.286,224.558)[lb]{\Large{\Black{$s$}}}
    \Text(74,183)[lb]{\Large{\Black{$\mathcal{M}_{0,4}$}}}
     \Text(110,151)[lb]{\Large{\Black{$\mathcal{M}_{1,3}$}}}
    \Text(74,115)[lb]{\Large{\Black{$\mathcal{M}_{0,2}$}}}
      \Text(145,115)[lb]{\Large{\Black{$\mathcal{M}_{2,2}$}}}
          \Text(108,82)[lb]{\Large{\Black{$\mathcal{M}_{1,1}$}}}
              \Text(178,82)[lb]{\Large{\Black{$\mathcal{M}_{3,1}$}}}
             \Text(30,82)[lb]{\Large{\Black{$\mathcal{M}_{-1,1}$}}}
                \Text(124,0)[lb]{\Large{\Black{$\mathcal{M}_{1,-1}$}}}
    \Text(50,3)[lb]{\Large{\Black{$\mathcal{M}_{-1,-1}$}}}
      \Text(70,55)[lb]{\Large{\Black{$\mathcal{M}_{0,0}$}}}
         \Text(227,141)[lb]{{\Blue{$\partial$}}}
        \Text(181,141)[lb]{{\Blue{$\LL$}}}
    \Text(227,161)[lb]{{\Blue{$\LL^{-1}$}}}
        \Text(181,161)[lb]{{\Blue{$\overline{\partial}$}}}
           \Text(173.457,55)[lb]{\Large{\Red{$\GE_2^*$}}}
    \Text(245.431,55)[lb]{\Large{\Red{$\GE_4$}}}
    \Text(110,124)[lb]{\Large{\Red{$\overline{\GE_2^*}$}}}
    \Text(110,191)[lb]{\Large{\Red{$\overline{\GE_4}$}}}
 \Text(134,92)[lb]{\Large{\Red{$\LL^{-1}$}}}   
 \Text(108,56)[lb]{\Large{\Red{$1$}}}   
 \Text(74,22)[lb]{\Large{\Red{$\LL$}}}   
   \Text(138,18)[lb]{\Large{\Red{$\GE_2^*\LL$}}}
 \Text(71,88)[lb]{\Large{\Red{$\overline{\GE_2^*}\LL$}}}     
  \SetWidth{0.5}
    \SetColor{Blue}
    \Line[dash,dashsize=7.197,arrow,arrowpos=0.5,arrowlength=4,arrowwidth=1,arrowinset=0.2](207.284,154.744)(241.832,189.291)
    \Line[dash,dashsize=7.197,arrow,arrowpos=0.5,arrowlength=4,arrowwidth=1,arrowinset=0.2](207.284,154.744)(172.737,189.291)
    \Line[dash,dashsize=7.197,arrow,arrowpos=0.5,arrowlength=4,arrowwidth=1,arrowinset=0.2](207.284,154.744)(241.832,120.196)
     \Line[dash,dashsize=7.197,arrow,arrowpos=0.5,arrowlength=4,arrowwidth=1,arrowinset=0.2](207.284,154.744)(172.737,120.196)
  \end{picture}
}
\end{center}

\vspace{0.1in}
The dashed arrows represent the action of $\LL, \LL^{-1}, \partial, \overline{\partial}$. Each solid circle represents a copy of $\mathcal{M}_{r,s}$ for $r+s$ even.  Some examples of modular forms are indicated in red.

\section{Primitives and obstructions}
In this section we study the equation 
\begin{equation} \label{Primeqn}
\partial F = f 
\end{equation} 
where $F, f \in \mathcal{M}$.  We say that $f \in \mathcal{M}$ has a modular $\partial$-primitive  if  $(\ref{Primeqn})$ holds for some $F$. 
We exhibit three obstructions for the existence of  modular primitives: the first is combinatorial, the second 
relates to modularity, and the third is arithmetic.

\subsection{Constants} Let us  view the operators $\partial_r, \overline{\partial}_s$ as (continuous) linear maps 
$$ \partial_r, \overline{\partial}_s   :   \Q[[q, \overline{q}]] [\LL^{\pm } ] \To\Q [[q, \overline{q}]][\LL^{\pm } ]  $$
of formal power series, setting aside questions of modularity for the time being. 

\begin{lem} \label{lemkers} The kernels of these maps are
\begin{eqnarray}  \label{kerpartial} \ker \partial_r  & =  & \LL^{-r} \Q [[ \overline{q}]] \\ 
 \ker \overline{\partial}_s & = & \LL^{-s} \Q [[ q]]\ . \nonumber
\end{eqnarray} 
In particular $\ker \partial_r \cap \ker \overline{\partial}_s$ vanishes if $r\neq s$ and is equal to  $\Q \Lef^{-r}$ if $r=s$. 
\end{lem} 
\begin{proof} Since $\partial_r \LL^k f = \LL^k \partial_{r+k} f$ $(\ref{partialLk})$, we can assume,  by multiplying by $\LL^{r}$ that $r=0$. 
The kernel of  $\partial_0 = (z-\overline{z}){ \partial \over\partial z} $ consists of antiholomorphic functions.  The second formula in $(\ref{kerpartial})$  is the complex conjugate of the first.
\end{proof}

We now consider the kernel of the operator $\partial$ acting on the space $\mathcal{M}$. 

\begin{prop}  \label{propModularKernel} Let  $F \in \mathcal{M}_{r,s}$ such that $\partial_r F =0$. Then 
$$\Lef^r F  \  \in  \   \overline{M}_{s-r} \ ,$$
where $\overline{M}_n$ denotes the space of anti-holomorphic modular forms of weight $n$.
In the case   $r> s$ i.e., `below the diagonal',   $F$ vanishes.   In the case $r=s$ we have
$$\ker \partial \cap \mathcal{M}_{r,r} = \C\,  \Lef^{-r}\ .$$ \end{prop}

\begin{proof} By lemma  \ref{lemkers}, we can write 
$ \Lef^r F  = \overline{g}$
where $ g: \HH \rightarrow \C$ is a holomorphic function.  
Since $f$ (respectively $\LL^r$) has weights $(r,s)$ (respectively $(-r,-r)$),  it follows that $ \overline{g}$ has weights $(0,s-r)$ and transforms like a modular form of weight $s-r$, i.e., 
$g( \gamma(z)) = (cz +d )^{s-r} g(z) $ for all $\gamma \in \SL_2(\Z)$ of the form $(\ref{gammaact})$. Thus $g \in M_{s-r}$. For the last part, use the well-known fact that there are no non-zero holomorphic modular forms of negative weight.
 \end{proof}

 Thus if equation $(\ref{Primeqn})$  has a solution,  it is unique up to addition  by an element  of $\C \Lef^{-r}$ if $h(F)=0$, and is unique if $h(F)>0$. 

\begin{cor}  Let $F \in \mathcal{M}_{r,s}$ and let  $f= \partial F $. There is a   solution $F' \in \mathcal{M}_{r,s}$ to (\ref{Primeqn}) 
whose antiholomorphic projection $p^a(F')$ vanishes. It is unique up to  addition by  a multiple of  $ \LL^{-r} \overline{\GE}_{s-r}$ for $s-r\geq 4$, where $\GE_{n}$ is the Eisenstein series (\ref{introEisdefn}).
\end{cor} 
\begin{proof} 
Since the Petersson inner product is non-degenerate, there exists a unique cusp form $g\in S_{s-r}$ such that $p^{a} (\overline{g}) = p^{a}(F)$.  Then 
$F' = F- (-2\pi)^r \LL^{-r} \overline{g}$ has the required properties. The second part follows since the  orthogonal complement of $S_{s-r}$ in $M_{s-r}$ is exactly the  vector space generated by the Eisenstein series.
\end{proof}

\subsection{Combinatorial obstructions}
The maps $\partial_r, \overline{\partial}_s$ are far from surjective.

\begin{lem} \label{lemvanishingcond} Suppose that  $f \in \C[[q, \overline{q}]][\LL^{\pm}]$ satisfies $f = \partial_r F$ for some $F \in \C[[q, \overline{q}]][\Lef^{\pm}]$. Then the coefficients in its expansion  $(\ref{qexp})$
 satisfy 
\begin{equation}  \label{avanishcond} 
a^{(-r)}_{0,n} = 0 \qquad \hbox{ for all } n\geq 0 \ .
\end{equation} 
\end{lem}
\begin{proof} Follows immediately from lemma $\ref{lempartialsqexp}$. 
\end{proof} 

This is not the only constraint: for every $m,n \geq 0$,  there is a condition on the  $a^{(k)}_{m,n}$, for varying $k$, in order  for $f$ to lie in the image of the map $\partial_r$.  Nonetheless,  $(\ref{avanishcond})$ is already sufficient to rule out the existence of primitives in many interesting cases. 

\begin{cor} \label{lemnoE2prim} There exists no element $F \in \C[[q, \overline{q}]][\Lef^{\pm}]$ satisfying 
$\partial_0 F = \LL  \GE^*_2$.
\end{cor} 
\begin{proof} By $(\ref{E2stardef})$, the $a^{(0)}_{0,0}$ term in $\Lef \GE^*_2 = \LL \GE_2 - {1 \over 4}$ is non-zero.  This violates $(\ref{avanishcond})$.
\end{proof}

\subsection{A condition involving  the pole filtration}

\begin{lem} If $f$ satisfies   the condition 
\begin{equation} \label{poleassumption} 
f \quad \in \quad P^{1-r} \mathcal{M}_{r+1,s-1}  \ ,
\end{equation} 
then  it  admits a combinatorial primitive
$F \in P^{-r} \C[[q,\overline{q}]][\LL^{\pm}]$
such that $\partial_r F = f$. 
\end{lem} 

\begin{proof} Denote the coefficients in the expansion of $f$ by $a^{(k)}_{m,n}$. By assumption they vanish for all  $k\leq -r$ and $k\geq N$ for some $N\geq -r$.  Denote  the coefficients of $F$ by $b^{(k)}_{m,n}$. Equation $(\ref{partialactformula})$ is equivalent to the set of equations
\begin{equation} \label{amnbmn} a_{m,n}^{(k)}  \quad =     \quad 2m \,b_{m,n}^{(k-1)} + (r+k)\, b_{m,n}^{(k)} \end{equation}
for every $m,n$.  Fix an $n$ and an  $m\geq 1$. Then, if we set $b_{m,n}^{(k)} =0$ for all $k\geq N$,  $(\ref{amnbmn})$ holds for all $k\geq N+1$. For $k=N$, we can solve it by setting 
$$ a_{m,n}^{(N)}  \quad  =     \quad 2m \,b_{m,n}^{(N-1)}\ . $$
Suppose we have  determined $b_{m,n}^{(k)}$ for all $k> K$. Then equation $(\ref{amnbmn})$ in the case $k=K+1$ can be solved uniquely for $b^{(K)}_{m,n}$ since $2m\neq 0$.   The process terminates at $k=1-r$, since for $k=-r$ equation $(\ref{amnbmn})$ reduces to:
$$ 0 =  a_{m,n}^{(-r)}  =     2m \,b_{m,n}^{(-r-1)}  +  0 \ . $$
Setting $b_{m,n}^{(k)}=0 $ for all $k<-r$, we therefore obtain a complete solution to $(\ref{amnbmn})$ for all values of $k$.  In the case $m=0$, 
the equations  $(\ref{amnbmn})$ can be solved trivially, provided that  $(\ref{avanishcond})$ holds. This is certainly implied by $(\ref{poleassumption})$.
\end{proof}

The commutation relation $\sfh = [\partial, \overline{\partial}]$ implies that 
$ h\, f + \overline{\partial} \partial f$ is in the image of $\partial$ for all $f \in \mathcal{M}$.
This remark, combined with  $(\ref{poleassumption})$, enables   
one to  prove the existence of combinatorial primitives
in many cases of interest.

\subsection{Obstructions from the  Petersson inner product} Another obstruction comes from the fact that a formal power series solution to $(\ref{Primeqn})$  is not necessarily modular. 

\begin{thm} \label{thmpartialsorthogtocusp} Let $f \in \mathcal{M}_{r,s}$. If  $f$ has a $\partial$-primitive in $ \mathcal{M}$ then 
\begin{equation} \langle f, g \rangle= 0  \qquad \hbox{ for all } \quad g \in S_{r-s} \hbox{ holomorphic} \  . \end{equation} 
In particular, $f$ is in the kernel of the holomorphic projection $(\ref{holprojection})$. 
\end{thm} 
\begin{proof}  By multiplying by $\LL^{r-1}$ and appealing to $(\ref{partialLk})$,  we see that the equation $\partial F' = f$ has a solution  if and only if there exists $F  \in \mathcal{M}_{0,s-r+2}$ such that 
$$\partial_0   (i \pi F) = \LL^{r-1} f\ .$$
From the definition of $\partial_0$, this implies that 
$${\partial F \over \partial z}  = \LL^{r-2} f\ .$$ 
Let $g\in S_{r-s}$ be  a holomorphic cusp form and consider the differential form 
$$\omega = F \overline{g} \, d\overline{z}$$
It is $\SL_2(\Z)$-invariant, since $F, \overline{g}, d \overline{z}$ are of weights $(0,s-r+2), (0,r-s)$, $(0,-2)$ respectively, and therefore their product is of type $(0,0)$. It  satisfies 
$$ d \omega = {\partial F \over \partial z}  \overline{g(z)} dz \wedge d  \overline{z} = \LL^{r-2} f(z) \overline{g(z) }   
\,   dz \wedge d \overline{z}  = (-2 \pi)^{r-2}    \, f(z) \overline{g(z)} y^r \, d \mathrm{vol} \ ,$$
which is also  of type $(0,0)$.  By Stokes' theorem we have 
$$\int_{\partial \mathcal{D}} \omega = \int_{\mathcal{D}} d \omega   =  (-2 \pi)^{r-2} \,  \langle f, g\rangle\ .$$
Consider the left-hand integral along the boundary  $\partial \mathcal{D}$ of the standard fundamental domain. By a classical argument using the modular invariance of $\omega$, it gives zero, since the contributions along the vertical line segments (from $\rho$ to $i \infty$ and $i \infty$ to $-\overline{\rho}$, where $\rho^2+\rho+1=0$) cancel due to translation-invariance $\omega(z) = \omega(z+1)$; the contributions along the segments of the circle $|z|=1$  from $\rho$ to $i$ and from $i$ to $-\overline{\rho}$ cancel due to $\omega(-z^{-1}) = \omega(z)$; and finally the contribution along a path  from $i\infty$ to $i \infty +1$, which corresponds to a  small loop in the $q$-disc, also gives zero because $g$ is cuspidal. 
\end{proof}  
The statement of the theorem can  formally be written
$\langle \partial F, g \rangle = 0$ if $g\in S$. 
By taking its  complex conjugate we also deduce that 
$\langle \overline{\partial} F, \overline{g} \rangle = 0$
for all $F \in \mathcal{M}_{r,s}$ and all cusp forms $g \in S_{s-r}$.
These equations can be written
$$p^h(\partial (F) ) =0 \qquad \hbox{ and } \qquad  p^a ( \overline{\partial} (F)   ) = 0 \qquad \hbox{ for all } F \in \mathcal{M} \ .$$  

\begin{cor} \label{cornocuspprim} For every non-zero cusp form $f \in S_{n}$, and every $k\in \Z$, the equation 
$\partial F = \LL^k f$ has no solution in $\mathcal{M}$. 
\end{cor} 
\begin{proof} By $(\ref{partialLk})$, we can  assume that $k=0$. If $F$ were to exist, the previous theorem with $g=f$ would imply that $0 = \langle f, f \rangle$.
But this  contradicts the fact that the Petersson inner product is positive-definite. 
\end{proof}

Primitives of cusp forms do exist if one allows poles at the cusp \S\ref{sectMercusp} and \cite{Mock}.

\subsection{Arithmetic obstructions}
Although this is largely irrelevant here, since we work mostly over the complex numbers, 
the equation $\partial F = f$ involves some subtle questions regarding the field of definition of the coefficients $a^{(k)}_{m,n}$. Fundamentally, complex conjugation is not rationally defined on algebraic de Rham cohomology.

For example, $\partial F = \GE_{4} \Lef$ has a unique solution given by a real analytic Eisenstein series $\mathcal{E}_{2, 0} \in \mathcal{M}_{2,0}$, to be defined in \S\ref{sectRAEisenstein},  but it has no solution with rational coefficients. This is because $\mathcal{E}_{2,0}$ involves the value of the Riemann zeta function $\zeta(3)$, which is  irrational as shown by Ap\'ery.
The examples of functions in $\mathcal{M}$ constructed in this paper arise from iterated integrals of modular forms, and their coefficients  $a^{(k)}_{m,n}$ are, in a certain sense,   periods. The period conjecture suggests that they are  transcendental.

\subsection{A class of modular iterated primitives} The  functions studied in this paper lie in a  special sub-class of functions inside $\mathcal{M}$ or $\mathcal{M}^!$. 
\begin{defn} Consider the largest space of functions 
$$\MI  \quad \subset \quad \bigoplus_{r,s\geq 0} P^{-r-s} \mathcal{M}_{r,s}  \ ,$$ 
equipped with an increasing `length' filtration $\MI_k \subset \MI$ such that $\MI_{k}=0$ if $k<0$,  and 
  every $F\in \MI_k$ satisfies
\begin{eqnarray} \label{MIdiffcond}
\partial F \quad & \in &  \quad \MI_k  \  + \  M[\LL] \times \MI_{k-1}   \\ 
\overline{\partial} F \quad & \in &  \quad \MI_k  \  +  \  \overline{M}[\LL] \times \MI_{k-1}  \nonumber 
\end{eqnarray} 
where $M$ (resp. $\overline{M}$) denotes the ring of holomorphic (anti-holomorphic) modular forms.  The  conditions $(\ref{MIdiffcond})$ are stable under the operation of taking sums of vector spaces, and therefore a largest  such space  exists and is unique. 
\end{defn}

By replacing $\MI$ with $\MI + \overline{\MI}$ we deduce that $\MI$ is closed under complex conjugation by  maximality.
This definition   is  computable: since
$\MI_k$ is contained  in the positive quadrant, equation $(\ref{MIdiffcond})$  implies that  any $F \in \MI_k$ of weights $(n,0)$ with $n\geq 0$  must satisfy
\begin{equation} \label{partialFonedge} \partial F \quad \in \quad M[\LL] \times \MI_{k-1}\ .
\end{equation}
In this region,  modular primitives are unique  by proposition $\ref{propModularKernel}$  (up to a possible constant when $n=0$).  
Then, for $F \in \MI_k$ of modular weights $(r,s)$ with $r\geq s$, the first equation of $(\ref{MIdiffcond})$ determines $F$ in terms of previously-determined functions by increasing induction on $s$.
The functions in the region $r<s$ are deduced by complex conjugation (or by  using  the second equation of $(\ref{MIdiffcond})$, starting from weights $(0,n)$). 

\begin{lem}  $\MI_0 = \C[\LL^{-1}]$.
\end{lem} 
\begin{proof}Since $\MI_{-1}=0$, any $F\in \MI_{0}$ of weights $(n,0)$ satisfies $\partial F=0$ by $(\ref{partialFonedge})$. If $n>0$, then $F$ vanishes by   proposition $\ref{propModularKernel}$.
Continuing in this manner, we see that any  $F\in \MI_{0}$ of weights $(r,s)$ for $r>s$ must also vanish, and in the case $r=s$, it must be of the form $F \in \C\LL^{-r}$.   Therefore $\MI_0 \subset \C[\LL^{-1}]$. Since $\partial \LL = \overline{\partial} \LL=0$,  the ring $\C[\LL^{-1}]$ indeed satisfies the conditions $(\ref{MIdiffcond})$ and hence $\MI_0  = \C[\LL^{-1}]$.
\end{proof}

\begin{rem}
There are some variants.  We can replace the space $M$ of holomorphic modular forms in the equations $(\ref{MIdiffcond})$ with another  space of modular forms $M'$ to define a  class of functions $\MI(M')$. Some examples:
\begin{enumerate}
  \vspace{0.05in}
\item Replace $M$ with $S$, the space of cusp forms. Since cusp forms do not admit modular primitives, one deduces by induction that   $\MI(S) =  \C [\LL^{-1}]$.

  \vspace{0.05in}
 \item Replace $M$ with 
\begin{equation} \label{Espacedefn} E = \bigoplus_{n\geq 2} \GE_{2n} \Q 
\end{equation} 
the $\Q$-vector space generated by Eisenstein series. 
  We obtain a space 
 $$\MI(E) \quad \subset \quad \MI \ .$$
 In  the sequel to this paper we  construct  a  subspace $\MI^E\otimes \C \subset \MI(E)$ (\S\ref{sectMIE}). We  show that $\MI(E)_k = \MI_k$ for $k=0,1$ but that they differ for $k=2$.
  
\end{enumerate}

\end{rem}
The class of functions $\MI$  has an interesting $\ssl_2$-module structure which could profitably  be reformulated in the language of  
\cite{ClassMaass}.

\subsection{Homological interpretations }  The following remarks can be skipped. Let $\mathcal{M}^{D+k} = \bigoplus_p \mathcal{M}_{p+k, p} $ denote  the subspace of $\mathcal{M}$ upon which $\sfh$ acts by $k$.  It is stable under multiplication by $\LL$.  Write $\mathcal{M}^D= \mathcal{M}^{D+0}$. 
Define an operator   
$$\partial^{D}  f  =    \partial(f) d z   + \overline{\partial}(f) d\overline{z}\ .$$
Since $dz$ and $d \overline{z}$ transform, respectively, like a modular form of weights $(-2,0)$ or $(0,-2)$,  it follows that $\partial^D$  defines a linear map of bidegrees $(-1,-1)$:
$$\partial^D : \mathcal{M}^D \To \mathcal{M}^{D+2} \, dz  \oplus  \mathcal{M}^{D-2} \, d\overline{z}$$
It extends  in the usual manner via the Leibniz rule to a linear map 
$$\partial^D :  \mathcal{M}^{D+2} \, dz  \oplus  \mathcal{M}^{D-2} \, d\overline{z} \To \mathcal{M}^D dz \wedge d\overline{z} \ . $$
It acts by $\partial^D( f dz + g d \overline{z}) =  ( \partial g - \overline{\partial} f    ) dz \wedge d\overline{z}$. 
It follows from the fact that   $[\partial, \overline{\partial}] = h$ vanishes    on $\mathcal{M}^D$, 
that these operators satisfy $(\partial^D)^2=0$.

\begin{defn} Define the diagonal complex to be 
$$0 \To \mathcal{M}^D  \overset{\partial^D}{\To} \mathcal{M}^{D+2} dz \oplus \mathcal{M}^{D-2} d \overline{z} \overset{\partial^D}{\To} \mathcal{M}^D dz \wedge d\overline{z} \To 0\ . $$
Denote its  cohomology groups to be $H^i(\mathcal{M}^D)$ for $i=0,1,2$.   They  inherit a grading in even degrees via the total weight grading on  $\mathcal{M}$, where $dz$ and $d \overline{z}$ have weight $-2$.  
\end{defn} 
It follows from lemma \ref{lemkers} that 
$H^0(\mathcal{M}^D) \cong \bigoplus_p  \Lef^{-p} \C$. In general, $H^i(\mathcal{M}^D)$ is a free graded $\C[\LL^{\pm}]$-module for all $i$.  
For example,  the one-form   of weight zero
$$\omega =   \GE_2^* dz    + \overline{\GE_2^*} d \overline{z}  $$
is closed and by lemma  \ref{lemnoE2prim}  defines a   
non-trivial cohomology class 
$[ \omega ]  
 \in H^1 (\mathcal{M}^D). $
 The obstructions to primitives discussed above can be interpreted in terms of this complex. For example, the 
 proof of theorem $\ref{thmpartialsorthogtocusp}$ can be interpreted as a functional:
 \begin{eqnarray} 
 \gr_{2m} H_{c}^2(\mathcal{M}^D) &\To&  \C \nonumber  \\ 
  \alpha dz \wedge d\overline{z} & \mapsto & \int_{\mathcal{D}} y^{-m}  \alpha dz \wedge d\overline{z}  \ .\nonumber 
 \end{eqnarray} 
 where  $H^2_c$ denotes the subspace of  $H^2$ representable by  forms in $S^{D} dz \wedge d\overline{z}$.   
 The obstructions in $H^1(\mathcal{M}^D)$ are purely combinatorial, by the following lemma:
  \begin{lem} \label{lemFmodular} Let $F \in \C [[q, \overline{q}]][\LL^{\pm}]$. If 
 $$\partial^D F \in \mathcal{M}^{D+2} dz \oplus \mathcal{M}^{D-2} d \overline{z}$$
 is modular, then so is $F$, i.e. $F\in \mathcal{M}^D$.
 \end{lem} 
\begin{proof} Suppose that $\partial F \in \mathcal{M}_{r+1,r-1}$ and $\overline{\partial}F \in \mathcal{M}_{r-1,r+1}$. Then for every $\gamma \in \mathrm{SL}_2(\Z)$,  $$ F(\gamma z) - (cz+d)^r (c\overline{z}+d)^r F(z)
 =C_{\gamma} \LL^{-r}$$
 where $C_{\gamma} \in \C$,  since by  lemma 
\ref{lemdeltaauto}, the left-hand side  lies in  $\ker \partial \cap \ker \overline{\partial}$. It follows that  $\gamma \mapsto C_{\gamma}\in \C$ is a cocycle.  Since $\SL_2(\Z)$ acts trivially on $\C$,   $Z^1 (\SL_2(\Z), \C)= \mathrm{Hom}(\SL_2(\Z);\C)$  vanishes, and therefore $C_{\gamma}=0$, i.e., $F$ is modular of weights $(r,r)$. 
\end{proof}

\section{Real analytic Eisenstein series} \label{sectRAEisenstein}
We consider in some detail the simplest possible family of non-holomorphic functions in $\mathcal{M}$
as a concrete  illustration.

\subsection{Modular primitives of  Eisenstein series}
Eisenstein series, unlike cusp forms, admit  modular primitives in $\mathcal{M}$.
Recall that  the real analytic Eisenstein series are defined for $\mathrm{Re}\, s>1$, $z\in\HH$ by the following function
$$E(z,s) = { 1\over 2} \sum_{m,n \neq (0,0)}  {y^s \over |mz+n|^{2s}}\ .$$

\begin{prop}
 For every $w \geq 1$, there exists a unique set of functions 
 $$\mathcal{E}_{r,s}  \quad \in \quad P^{-w}\,  \mathcal{M}_{r,s}$$
 with  $r,s\geq0 $ and $r+s= w$,  which satisfy the following equations:
\begin{eqnarray} \label{realanalholequation} 
 \partial \, \mathcal{E}_{w,0}   &= & \LL \GE_{w+2}    \\ 
 \partial \,\mathcal{E}_{r,s} -(r+1) \mathcal{E}_{r+1, s-1}  &= &  0   \qquad \qquad  \hbox{ for all  }1\leq s\leq w \  \nonumber 
\end{eqnarray} 
and 
\begin{eqnarray} \label{realanalantiholequation} 
\overline{\partial} \, \mathcal{E}_{0,w}   &= & \LL  \overline{\GE}_{w+2}    \\ 
 \overline{\partial} \,\mathcal{E}_{r,s} - (s+1) \mathcal{E}_{r-1, s+1}   &= &   0  \qquad  \qquad \hbox{ for all  } 1\leq r \leq w \ .  \  \nonumber 
\end{eqnarray} 
These functions can be given explicitly by the following formula
\begin{equation} \label{Eijelementary} 
\mathcal{E}_{r,s} (z)=  { w! \over   ( 2 \pi i )^{w+2}}  {1 \over 2}   \sum_{m,n}   { \LL  \over (mz+n)^{r+1} (m\overline{z}+n)^{s+1} } 
\end{equation} 
where the sum is over all  integers $m,n \in \Z^2 \backslash (0,0)$. 

\end{prop} 
\begin{proof}  The uniqueness follows from proposition \ref{propModularKernel}.  For the existence, 
formula $(\ref{Eijelementary})$  converges and defines a modular function of weights $r,s$.
We must verify $(\ref{realanalholequation})$ and $(\ref{realanalantiholequation})$. These follow  from  
the  following  identity, which, holds for any integers $r,s$:
$$\partial_r \Big( {z- \overline{z}  \over (mz+n)^{r+1} (m\overline{z}+n)^{s+1}}\Big) = (r+1)  \Big(   {z- \overline{z}  \over (mz+n)^{r+2} (m\overline{z}+n)^{s}}  \Big) \ .$$
By taking the complex conjugate, we deduce a similar formula for $\overline{\partial}$ on interchanging $r$ and $s$. 
It follows from the definition of the holomorphic Eisenstein series  as a sum:
$$ \GE_{w+2}  = {(w+1)!  \over (2\pi i)^{w+2}} {1 \over 2}   \sum_{(m,n)\neq (0,0)} {1 \over (mz+n)^{2w+2}}$$
 that $\mathcal{E}_{w,0}$ satisfies the first equation of 
 $(\ref{realanalholequation})$. The first equation of $(\ref{realanalantiholequation})$ follows by conjugating. 
It remains to verify the expansion $(\ref{qexp})$.  For this, note that   
 for all $m\geq 1$ the definition $(\ref{Eijelementary})$ implies the identity
\begin{equation} \mathcal{E}_{m,m}  =  { i \over ( 2 i \pi)^{2m+1}} { (2m)! \over y^{m} }  E(z,m+1)\ .
\end{equation}
The  expansion of the right-hand side is well-known and lies in $\C[[q, \overline{q}]][\LL^{\pm}]$. The expansions of the   functions $\mathcal{E}_{r,s}$ are deduced from $\mathcal{E}_{m,m}$ by applying $\partial, \overline{\partial}$.
 \end{proof}

We immediately deduce the following properties: 
\begin{cor}  \label{corEijbasic} For all $r+s= w>0$,
the functions $\mathcal{E}_{r,s}$ satisfy $\overline{\mathcal{E}}_{r,s}  =    \mathcal{E}_{s,r}$, 
\begin{eqnarray} 
\Delta \mathcal{E}_{r,s} & = & -  w \, \mathcal{E}_{r,s} \ , \nonumber \\
 p( \mathcal{E}_{r,s} ) & = & 0  \ , \nonumber 
\end{eqnarray} 
where $p= p^{h}+p^{a}$ denotes the holomorphic and antiholomorphic projections. 
\end{cor} 
\begin{proof} The compatibility with complex conjugation follows by symmetry of $(\ref{realanalholequation})$ and $(\ref{realanalantiholequation})$ and uniqueness. 
The Laplace equation  follows from $(\ref{LaplaceDef})$, $(\ref{realanalholequation})$ and $(\ref{realanalantiholequation})$. 
The last equation follows from theorem $\ref{thmpartialsorthogtocusp} $ since  $\mathcal{E}_{r,s}$ is in the image of either $\partial$ or $\overline{\partial}$. 
\end{proof}

\begin{prop} The constant part of $\mathcal{E}_{r,s}$ is given by 
\begin{equation}  \label{constantpartofRealEis}
\mathcal{E}_{r,s}^0 =  {  - B_{w+2} \over  2(w+1)(w+2)} \LL      +      {(-1)^s \over 2} {w! \over 2^w}\binom{w}{r}     \zeta(w+1)  \LL^{-w}   \end{equation} 
where   $w= r+s >0$ is even. Furthermore,  $\mathcal{E} - \mathcal{E}_{r,s}^0$ has rational coefficients. 
\end{prop} 
\begin{proof}  The statement is well-known for $r=s=w$, since it reduces to the Fourier expansion of the real analytic Eisenstein series $E(z,w+1)$. 
The remaining cases are deduced by applying $\partial$  via $(\ref{realanalholequation})$ and by $\mathcal{E}_{r,s} = \overline{\mathcal{E}}_{s,r}$. 
An alternative way to prove this theorem is to use the expression for $\mathcal{E}_{r,s}$ as the real part of the single iterated integral of holomorphic
Eisenstein series \cite{MMV} \S8, and use the computation of the cocycle of the latter \cite{MMV}, lemma 7.1, to write down the constant terms directly. See \S \ref{sectEisasint}. 
\end{proof} 

\subsection{Explicit formulae}
For all $w\geq 1$ write
$$g_{2w+2}^{(k)}(q)=  (-1)^k k! \sum_{n\geq 1} { \sigma_{2w+1}(n) \over (2n)^{k+1}}  q^n \ .$$
Then for any $a+b=2w$, define 
\begin{equation} 
R_{a,b}= (-1)^{a} \binom{2w}{a} \sum_{k=b}^{a+b}  \binom{a}{k-b}    { g_{2w+2}^{(k)}(q)  \over \LL^k} \ .
\end{equation}
Then the real analytic Eisenstein series are given explicitly by
$$\mathcal{E}_{a,b} =  \mathcal{E}^0_{a,b}+   R_{a,b}   + \overline{R}_{b,a}\ ,$$
where $ \mathcal{E}^0_{a,b}$ is $(\ref{constantpartofRealEis})$.
This formula  in the case $a=b$ is equivalent to  the known Fourier expansion  of the real analytic Eisenstein 
series. One can verify  the other cases by checking that they satisfy the differential equations   $(\ref{realanalholequation})$ and $(\ref{realanalantiholequation})$. See \cite{Mock} for details.

 \subsection{Description of $\MI_1$} We already showed that $\MI_0 = \C[\LL^{-1}]$.

\begin{cor}   In length one, 
$$\MI_1  =   \MI_0 \otimes_{\C} \bigoplus_{r,s\geq 0 , r+s \geq 2 } \C \,\mathcal{E}_{r,s} $$
\end{cor}
\begin{proof}
Let $F\in \MI_{1}$ of weights $(n,0)$. By $(\ref{partialFonedge})$ it satisfies 
$\partial F   \in     M \LL$. 
 Since $\partial F$ is orthogonal to cusp forms by theorem $\ref{thmpartialsorthogtocusp}$, 
 it must satisfy $\partial F \in \C \GE_{n+2} \LL$.  This equation has the unique family of solutions  $F \in \C \mathcal{E}_{n,0}$. 
 By equation $(\ref{MIdiffcond})$, the elements $F\in \MI_{1}$ of weights $(r,s)$ with $r>s$ are iterated primitives of real analytic Eisenstein series and modular forms $M[\LL]$, and hence also real analytic Eisenstein series, by a similar argument.
 We conclude that  $\MI_1$ is contained in the $\C[\LL^{-1}]$-module generated by the $\mathcal{E}_{r,s}$. Since the latter satisfy
 $(\ref{MIdiffcond})$, this proves  equality.
 \end{proof}

\subsection{Picture of the real analytic Eisenstein series}
Based on the previous picture of $\mathcal{M}$, the real analytic Eisenstein series can be viewed as follows:
\\

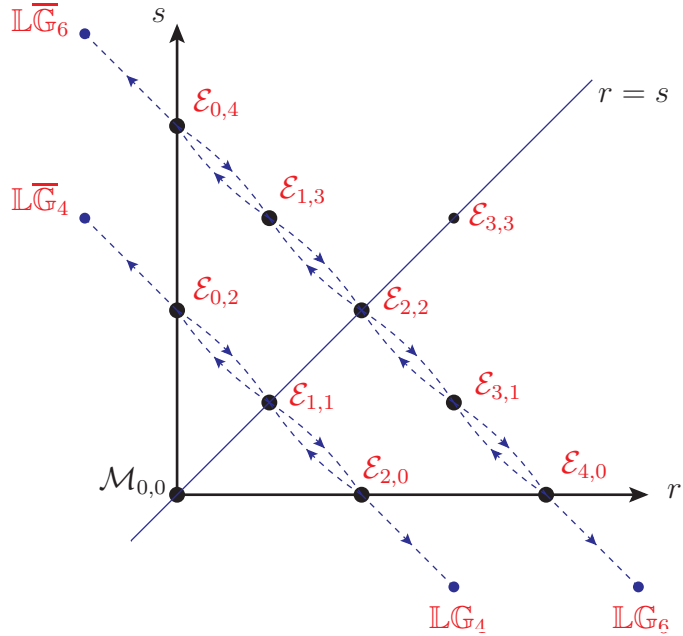
\begin{center}
\fcolorbox{white}{white}{
  \begin{picture}(290,239) (36,-17)
    \SetWidth{1.0}
    \SetColor{Black}
    \Line[arrow,arrowpos=0,arrowlength=5,arrowwidth=2,arrowinset=0.2,flip](104.308,204.994)(104.308,31.147)
    \Line[arrow,arrowpos=1,arrowlength=5,arrowwidth=2,arrowinset=0.2](104.308,31.147)(278.154,31.147)
   \SetWidth{2.0}
    \Vertex(104.308,31.147){3}
    \Vertex(139.077,65.917){3}
    \Vertex(104.308,100.686){3}
    \Vertex(104.308,170.224){3}
    \Vertex(139.077,135.455){3}
    \Vertex(173.846,31.147){3}
    \Vertex(173.846,100.686){3}
    \Vertex(208.615,65.917){3}
    \Vertex(243.385,31.147){3}
    \Vertex(208.615,135.455){2.049}
      \SetWidth{1.0}
    \SetWidth{0.5}
      \SetColor{Blue}
    \Arc[dash,dashsize=2.173,arrow,arrowpos=0.5,arrowlength=2.167,arrowwidth=0.9083,arrowinset=0.2,clock](203.907,235.054)(118.84,-123.06,-146.94)
    \Arc[dash,dashsize=2.173,arrow,arrowpos=0.5,arrowlength=2.167,arrowwidth=0.9083,arrowinset=0.2,clock](40.057,71.204)(118.038,57.022,32.978)
    \Arc[dash,dashsize=2.173,arrow,arrowpos=0.5,arrowlength=2.167,arrowwidth=0.9083,arrowinset=0.2,clock](238.676,200.285)(118.84,-123.06,-146.94)
    \Arc[dash,dashsize=2.173,arrow,arrowpos=0.5,arrowlength=2.167,arrowwidth=0.9083,arrowinset=0.2,clock](74.826,36.435)(118.038,57.022,32.978)
    \Arc[dash,dashsize=2.173,arrow,arrowpos=0.5,arrowlength=2.167,arrowwidth=0.9083,arrowinset=0.2,clock](273.445,165.516)(118.84,-123.06,-146.94)
    \Arc[dash,dashsize=2.173,arrow,arrowpos=0.5,arrowlength=2.167,arrowwidth=0.9083,arrowinset=0.2,clock](308.215,130.747)(118.84,-123.06,-146.94)
    \Arc[dash,dashsize=2.173,arrow,arrowpos=0.5,arrowlength=2.167,arrowwidth=0.9083,arrowinset=0.2,clock](238.676,130.747)(118.84,-123.06,-146.94)
    \Arc[dash,dashsize=2.173,arrow,arrowpos=0.5,arrowlength=2.167,arrowwidth=0.9083,arrowinset=0.2,clock](203.907,165.516)(118.84,-123.06,-146.94)
    \Arc[dash,dashsize=2.173,arrow,arrowpos=0.5,arrowlength=2.167,arrowwidth=0.9083,arrowinset=0.2,clock](40.057,1.666)(118.038,57.022,32.978)
    \Arc[dash,dashsize=2.173,arrow,arrowpos=0.5,arrowlength=2.167,arrowwidth=0.9083,arrowinset=0.2,clock](109.595,1.666)(118.038,57.022,32.978)
    \Arc[dash,dashsize=2.173,arrow,arrowpos=0.5,arrowlength=2.167,arrowwidth=0.9083,arrowinset=0.2,clock](144.365,-33.103)(118.038,57.022,32.978)
    \Arc[dash,dashsize=2.173,arrow,arrowpos=0.5,arrowlength=2.167,arrowwidth=0.9083,arrowinset=0.2,clock](74.826,-33.103)(118.038,57.022,32.978)
    \Line(86.923,13.763)(260.769,187.609)
    \SetWidth{0.5}
    \Vertex(208.615,-3.622){2.049}
    \Vertex(278.154,-3.622){2.049}
    \Vertex(69.538,135.455){2.049}
    \Vertex(69.538,204.994){2.049}
    \Text(42,137)[lb]{\Large{\Red{$\LL \overline{\GE}_4$}}}
    \Text(42,205)[lb]{\Large{\Red{$\LL \overline{\GE}_6$}}}
    \Text(200,-21)[lb]{\Large{\Red{$\LL \GE_4$}}}
    \Text(269.461,-21)[lb]{\Large{\Red{$\LL \GE_6$}}}
    \Text(110.827,173)[lb]{\Large{\Red{$\mathcal{E}_{0,4}$}}}
    \Text(143.423,140)[lb]{\Large{\Red{$\mathcal{E}_{1,3}$}}}
    \Text(182.538,98)[lb]{\Large{\Red{$\mathcal{E}_{2,2}$}}}
      \Text(215,130)[lb]{\Large{\Red{$\mathcal{E}_{3,3}$}}}
    \Text(217.308,66)[lb]{\Large{\Red{$\mathcal{E}_{3,1}$}}}
    \Text(249.904,36)[lb]{\Large{\Red{$\mathcal{E}_{4,0}$}}}   
    \Text(110.827,102)[lb]{\Large{\Red{$\mathcal{E}_{0,2}$}}}
    \Text(147,62)[lb]{\Large{\Red{$\mathcal{E}_{1,1}$}}}
    \Text(175,34)[lb]{\Large{\Red{$\mathcal{E}_{2,0}$}}}
    \Text(75,30)[lb]{\Large{\Black{$\mathcal{M}_{0,0}$}}}
    \Text(290,30)[lb]{\Large{\Black{$r$}}}
    \Text(95,210)[lb]{\Large{\Black{$s$}}}
    \SetWidth{0.5}
    \Line[dash,dashsize=2.173,arrow,arrowpos=0.5,arrowlength=2.167,arrowwidth=1,arrowinset=0.2](104.308,170.224)(69.538,204.994)
    \Line[dash,dashsize=2.173,arrow,arrowpos=0.5,arrowlength=2.167,arrowwidth=1,arrowinset=0.2](104.308,100.686)(69.538,135.455)
    \Line[dash,dashsize=2.173,arrow,arrowpos=0.5,arrowlength=2.167,arrowwidth=1,arrowinset=0.2,flip](208.615,-3.622)(173.846,31.147)
    \Line[dash,dashsize=2.173,arrow,arrowpos=0.5,arrowlength=2.167,arrowwidth=1,arrowinset=0.2,flip](278.154,-3.622)(243.385,31.147)
     \Text(264,180)[lb]{\Large{\Black{$r=s$}}}
  \end{picture}
}
\end{center}
\vspace{0.1in}
The dashed arrows going up and down the anti-diagonals are $\partial$ and $\overline{\partial}$. 
The classical real analytic Eisenstein series are the functions $\mathcal{E}_{n,n}$ lying along the diagonal $r=s$.

\section{Eigenfunctions of the Laplacian} 

The results of this section are not  needed for the rest of the paper. We show that the space $\mathcal{M}$   has   very limited  overlap  with  the  theory of Maass waveforms \cite{Maass}, and determine to what extent the solutions to a Laplace eigenvalue equation are not unique.

Call  $F\in \mathcal{M}$  an eigenfunction of $\Delta$ if there exists $\lambda \in \C$, the eigenvalue, such that  $\Delta  F = \lambda F$. It decomposes into a sum of terms $F_{r,s}  \in \mathcal{M}_{r,s}$  satisfying $\Delta_{r,s} F = \lambda F$.  

\begin{thm} \label{thmLaplaceEigen} Let $F$ be an eigenfunction of the Laplacian. Then its  eigenvalue is an integer, and $F$ is a linear combination over $\C[\LL^{\pm}]$ of real analytic Eisenstein series $\mathcal{E}_{r,s}$,  almost holomorphic  modular forms and their complex conjugates.\end{thm}

Let us write $\mathcal{HM} \subset \mathcal{M}$ to denote the space of Laplace eigenfunctions. 
 It follows from lemma  \ref{lemoperatoridentities} 
   that it is stable under   the action of $\Or= \Q[\LL^{\pm}][\partial, \overline{\partial}]$.  Furthermore, the subspace $\mathcal{HM}(n)$ of eigenfunctions with eigenvalue $n$ is stable under the action of the Lie algebra $\ssl_2$ generated by $\partial, \overline{\partial}$.

Every  holomorphic modular form $f\in M_{n}$ lies in $\mathcal{HM}(0)$ since
$\Delta f = - \partial_{n-1} \overline{\partial}_0   f = 0$.  The same is true of $\mm$ defined in  \S\ref{remalmosthol}. More generally, $\LL^k f$ is an eigenfunction with eigenvalue $(n-k-1)k$.   Since the ring of almost holomorphic modular forms is generated by holomorphic modular forms and $\mm$ by the action of $\partial$, it follows that any almost holomorphic (or anti-holomorphic) modular form  lies in $\mathcal{HM}$.
\subsection{Proof of theorem $\ref{thmLaplaceEigen}$.}

\begin{lem} \label{lemEfunctionExp} Let $F \in \mathcal{M}_{r,s}$ such that  $\Delta_{r,s} F = \lambda F$. Then 
there exists an integer $k_0$ such that $\lambda = - k_0 (k_0+w-1)$, where $w=r+s$ is the total weight. We can assume $k_0= \min\{ k_0, 1-w-k_0\}$. 
Then $F$ is of the form 
\begin{equation} \label{EfunctionExp}  F=  \alpha \,\LL^{k_0} + \beta \,\LL^{1-w-k_0} + \sum_{k_0 \leq k\leq -s} \LL^k f_k(q) + \sum_{k_0 \leq k\leq -r} \LL^k g_k(\overline{q})  
\end{equation}
where $\alpha, \beta \in \C$, and  $f_k(q) \in \C[[q]]$,  $g_k( \overline{q}) \in \C[[\overline{q}]]$ have no constant terms. 
\end{lem} 

\begin{proof}  Assume that $F$ is non-zero and denote the coefficients in its expansion $(\ref{qexp})$  by 
$a_{m,n}^{(k)}$.
We first show that  $a^{(k)}_{m,n}=0$ if $mn\neq 0$. Fix   $m,n$ such that $a^{(k)}_{m,n}\neq 0 $ for some $k$. Choose
$k$ maximal with this property. Taking the coefficient of $\LL^{k+2} q^m \overline{q}^n$ in  the equation $\Delta_{r,s} F = \lambda F$ implies, via $(\ref{Deltars})$,   that 
$ \lambda a^{(k+2)}_{m,n} =  - 4mn\, a^{(k)}_{m,n} $, which implies that $mn =0$. Therefore all $a^{(k)}_{m,n}$ vanish for $mn\neq 0$.  
Now, for any $m,n$, choose $k$ minimal such that $a^{(k)}_{m,n}$ is non-zero. Equation $(\ref{Deltars})$ implies that 
$\lambda a^{(k)}_{m,n} = - k (k + w- 1) a^{(k)}_{m,n}$, which proves the first part of the lemma. 
 The equation $x^2+ x(w-1) + \lambda=0$ has two integral solutions $k_0$ and $1-w-{k_0}$,  which are distinct since $w$ is even. The assumption that $k_0$ is the smaller of the two 
 implies that $a^{(k)}_{m,n}$ vanishes for all $k< k_0$. 

 Now consider a non-zero coefficient of the form $a^{(k)}_{m,0}$ with $m\neq 0$. Let $k$ be maximal. Equation $(\ref{Deltars})$ implies that $\lambda a^{(k+1)}_{m,0} = 2 m(k+s) a^{(k)}_{m,n} - k (k+w-1) a^{(k+1)}_{m,0}$, which implies that $m(k+s)=0$ since $a^{(k+1)}_{m,0} =0$. Therefore $k=-s$.  A similar computation with terms of the form $a^{(k)}_{0,n}$ shows that they all vanish if $k> -r$. 
It remains to determine the constant terms $a^{(k)}_{0,0}$. Equation $(\ref{Deltars})$ implies that $\lambda a^{(k)}_{0,0} = - k (k+w-1) a^{(k)}_{0,0}$, so  by the above  $a^{(k)}_{0,0}$ is non-zero only for  $k\in \{ k_0, 1-w-{k_0}\} $. 
\end{proof}

\begin{lem} Let $F\in \mathcal{M}_{r,s}$ be an eigenfunction of the Laplacian. Then there exist  integers $M, N\geq 0 $ such that 
$\overline{\partial}^M \partial^N F \in \C[\LL^{\pm}]$. \end{lem} 
\begin{proof} Apply $\partial_r$ to the expansion $(\ref{EfunctionExp})$.  By lemma $\ref{lemkers}$, 
this annihilates the term $\LL^{k} g_{k}(\overline{q})$ for $k=-r$.  The  terms of the form $\LL^{k} g_{k}(\overline{q})$ are simply multiplied by $k+r$. Its action on terms of the form
$\LL^k f_k(q)$ increases the degree in $\LL$ by at most one, by
  by (\ref{partialactformula}). Therefore $\partial_r F $ has a similar expansion to $(\ref{EfunctionExp})$, with $(r,s)$ replaced by $(r+1, s-1)$.  Applying $\partial_{r-1}$  kills the term $\LL^{k} g_{k}(\overline{q})$ for with $k=1-r$. Proceeding in this manner, every term of the form 
$\LL^{k} g_{k}(\overline{q})$ is   eventually  annihilated (this also follows directly from lemma  \ref{lemEfunctionExp}  since  $\partial^m F$ are  eigenfunctions of the Laplacian with the same eigenvalue $\lambda$ as $F$).   Now by a similar argument, repeated application of $\overline{\partial}$ annihilates all  the terms of the form $\LL^k f_k(q)$. 
\end{proof}

\begin{lem} The maps  $\overline{\partial}: \widetilde{M} \rightarrow \widetilde{M}$ and 
$\partial : \widetilde{\overline{M}} \rightarrow \widetilde{\overline{M}}$  
are surjective.  
\end{lem} 
\begin{proof}
Since $\overline{\partial} \,\mm =1$,  any element $ f \mm^i$, where $i\geq 0$ and  $f\in M[\LL^{\pm}]$, is the $\overline{\partial}$-image  of
$(i+1)^{-1} f \mm^{i+1}$. The second statement follows by complex conjugation. 
\end{proof}

\begin{lem} Consider the linear map  
$ \partial :  \widetilde{M} [\LL^{\pm}]  \rightarrow  \widetilde{M}[\LL^{\pm}]$. Then $\ker {\partial} \cong \C$ and 
$$\mathrm{Coker}\, \partial \cong M[\LL^{\pm}] \oplus \C \mm[\LL^{\pm}] \ .$$
\end{lem}
\begin{proof}  The statement about the kernel follows immediately from lemma $\ref{lemkers}$. It follows from  the calculations in \S\ref{remalmosthol},  that  for any $f\in M_{n}$ and $k \geq 0$, 
$$\partial \mm^k f  = (-k-n) \mm^{k+1} f   \quad + \quad  \hbox{terms of degree } \leq k \hbox{ in } \mm\ .$$
Since $\LL$ commutes with $\partial$,  all terms of  the form $ f \mm^k \LL^r$, where $f\in M_n$, are in the image of $\partial$ whenever $k\geq2$ or $k=1$ and $n>0$. Conclude using $\widetilde{M}[\LL^{\pm}]= M[\mm, \LL^{\pm}]$.  
\end{proof}

\begin{cor} Let $V \subset \mathcal{M}$ denote the $\C[\LL^{\pm}]$-module   generated by the real analytic Eisenstein series
$\mathcal{E}_{r,s}$,   $\widetilde{M}$ and  $\widetilde{\overline{M}}$. 
If $F \in \mathcal{M}$ satisfies  $\partial F \in V$, then  $F \in V$.  By complex conjugation, the same statement holds with $\partial$ replaced with $\overline{\partial}$. 
\end{cor} 

\begin{proof}
 By  proposition
\ref{realanalholequation}, the Eisenstein series $\GE_{2n} \LL^k$, for $n\geq 2$ and the functions $\mathcal{E}_{r,s}$ with $r>0$ admit  $\partial$-primitives  in  $V$. By   the  above,  we can assume that $\partial F$ is a linear combination of 
$$ \mm \LL^k  \quad  , \quad f \LL^k \quad , \quad \mathcal{E}_{0,2n} \LL^k $$
where $f$ is a cusp form. Since these elements have distinct $\sfh$-degrees, we can treat each case in turn,  by linearity.
But we showed in    corollary \ref{lemnoE2prim}  that $\mm \LL^k$ has no $\partial$-primitive in $\mathcal{M}$, and likewise,  in corollary \ref{cornocuspprim}  that cusp forms have no primitives either. The elements $\mathcal{E}_{0,2n}$ (and hence $\LL^k \mathcal{E}_{0,2n}$) have
no modular primitives  by lemma  \ref{lemvanishingcond},  since the coefficient of  $\LL$ in $\mathcal{E}^0_{0,2n}$ 
is non-zero  by $(\ref{constantpartofRealEis})$. Therefore, none of these cases can arise, and we conclude that if $\partial F \in V$, so too is $F\in V$.
\end{proof} 

An eigenfunction of the Laplacian $F$ satisfies   $\overline{\partial}^M \partial^N F \in \C[\LL^{\pm}]  \subset V$. It follows from the previous corollary and induction  on $N$ that $F \in V$. 
This completes the proof. 
\begin{rem}
In passing, we have shown that the ring of almost holomorphic modular forms $M[\mm, \Lef^{\pm}]$ is  the subspace of functions $f\in \mathcal{M}$ such that $a_{m,n}^{(k)}(f)=0$ for all $n>0$, or equivalently, which satisfy  $\overline{\partial}^N f=0$ for sufficiently large $N$. 
\end{rem}

\section{Mixed Rankin-Cohen brackets} \label{sectRC}
This section can be skipped. 
Any operator in  
$\Or=\Q[\LL^{\pm}][\partial, \overline{\partial}]$ can be expressed as a polynomial in $ \LL^{\pm}, {\partial \over \partial z}, {\partial \over \partial {\overline{z}}}$. We wish  to find elements of  $\Or \otimes \Or$, which act via 
$$\Or \otimes \Or :  \mathcal{M} \otimes \mathcal{M} \To \mathcal{M}\ , $$
which are homogeneous in $\LL$ when expressed in terms of $\partial\over \partial z$ and $\partial \over \partial \overline{z}$.  Since these operators will not be used in this paper,  we shall only illustrate  how the theory of Rankin-Cohen brackets can be recovered in  some basic examples and leave the many possible extensions to the reader. 

\begin{example}  (Operators of order 1).  Starting with the four operators  given by $\partial \otimes \id$, $\id \otimes \partial$, and their complex conjugates, we form  the general  operator:
$$ a_1 \partial \otimes \id + a_2 \id \otimes \partial + a_3 \overline{\partial} \otimes \id + a_4  \id \otimes \overline{\partial}\ , $$
where $a_i \in \Q$. It acts upon $f \otimes g \in \mathcal{M}_{r_1,s_1} \otimes \mathcal{M}_{r_2,s_2}$ by 
 $$a_1 \big( \LL' {\partial f \over \partial z}  + r_1 f\big) g  + a_2 f \big( \LL' {\partial g \over \partial z}  + r_2 g\big)+ a_3 \big( - \LL' {\partial f \over \partial \overline{z}}  + s_1 f\big)g + a_4 f \big( - \LL' {\partial g \over \partial \overline{z}}  + s_2 g\big) $$ 
  where $\LL' i \pi = \LL$. The terms of degree zero in $\LL'$ vanish if and only if 
  $$a_1 r_1 + a_2 r_2 +a_3 s_1 +a_4 s_2 =0 \ .$$ 
  A basis for its solutions are  $(r_2,-r_1,0,0)$,  $( s_1,0,-r_1,0)$ and $(0,0,-s_2,s_1)$. 
Dividing by $\LL'$, the first and third solutions yield  the  combinations:
\begin{eqnarray}  \label{RC1}
[f,g]_1  &=&  r_2 {\partial f \over \partial z} g - r_1 f {\partial g \over \partial z}   \\
{[}f,g]_{\overline{1}}  &=&  s_2 {\partial f \over \partial\overline{ z}} g - s_1 f {\partial g \over \partial \overline{ z}}   \nonumber
\end{eqnarray}
which are the first Rankin-Cohen bracket  and its complex conjugate. The second solution defines an additional  element 
$ (D f) g$  of mixed weights, where 
\begin{equation} \label{Dfmixed} Df:  =  {1 \over \LL} \big( s_1 \partial_{r_1} - r_1 \overline{\partial}_{s_1}\big)  f = s_1 {\partial f \over \partial z}  + r_1 {\partial f \over  \partial \overline{z}} \qquad \in \qquad \mathcal{M}_{r_1+2,s_1} \oplus \mathcal{M}_{r_1,s_1+2} 
\end{equation}
It splits into two components of different  modular weights in the algebra $\mathcal{M}$.  For example,  $\frac{\partial f}{\partial z}  \in \mathcal{M}_{n+2,0} \oplus \mathcal{M}_{n,2}$ for any holomorphic modular form $f \in M_n$.
\end{example} 
The properties  of the brackets $(\ref{RC1})$ are d  well-known. For instance, the bracket is anti-symmetric and satisfies the Jacobi identity  \cite{Zag123} \S5.2.

\begin{example} (Operators of order 2). We can easily extend this analysis to operators of higher order.  We exclude mixed terms such as $(\ref{Dfmixed})$. The operators of order 2 and bidegrees $(2,-2)$ are of the form 
$$\partial^2 \otimes \id \quad , \quad  \partial \otimes \partial  \quad , \quad  \id \otimes \partial^2$$
A similar analysis to the one above produces a one dimensional family of linear combinations which are homogeneous of degree two in $\LL'$. They are generated by  the second order Rankin-Cohen bracket, defined for $ f\in \mathcal{M}_{r_1,s_1}$ and $g \in \mathcal{M}_{r_2,s_2}$ by
\begin{equation} 
[f, g]_2 = { r_2 (r_2+1)  \over 2} {\partial^2 f \over \partial z^2} g - (r_1+1)(r_2+1) {\partial f \over \partial z} { \partial g \over \partial z} +  { r_1 (r_1+1)  \over 2} f {\partial^2 g \over \partial z^2} \ .
\end{equation}
In bidegrees $(-2,2)$, we obtain  the complex conjugate bracket. A new feature  appears in bidegree $(0,0)$.  Indeed, consider the following  five terms of 
this type: 
$$\partial \overline{\partial} \otimes \id \quad , \quad \partial \otimes \overline{\partial} \quad , \quad \overline{\partial} \otimes \partial \quad , \quad \id \otimes \partial \overline{\partial} \quad , \quad \id \otimes \id\ . $$
The commutation relation $(\ref{partialcommutators})$ implies that $\partial \overline{\partial}$,  $\overline{\partial } \partial $ and $\id$ are linearly related so there are exactly five such operators. The linear combination 
$$ r_2s_2\, \partial \overline{\partial} \otimes \id  - s_1 r_2 \, \partial \otimes \overline{\partial}  - s_2r_1 \, \overline{\partial} \otimes \partial + r_1s_1 \, \id \otimes \partial \overline{\partial} + r_1r_2(s_1+s_2) \id \otimes \id$$
generates a one-dimensional family of operators which become homogeneous in $\LL'$ after rewriting them in terms of ${\partial \over \partial z}$ and $\partial \over \partial \overline{z}$. The coefficient of $(\LL')^2$ is the quantity
\begin{equation} \Big(f, g \Big)_2 = s_1r_2 \, {\partial f \over \partial z}  {\partial g \over \partial \overline{z}}  + s_2 r_1 \,   {\partial f \over \partial \overline{z}} {\partial g \over \partial z}  - r_1 s_1 \, f {\partial^2 g \over \partial z \partial \overline{z}} -r_2s_2 \, g {\partial^2 f \over \partial z \partial \overline{z}} \ , 
\end{equation}
which is symmetric in $f$ and $g$ and is an element of $\mathcal{M}_{r_1+r_2, s_1+s_2}$. 
It can be written more elegantly as  a composition of operators, as follows:
$$\Big(f, g \Big)_2 =  \big( {\partial_z \otimes r_2} - r_1 \otimes \partial_{z}\big) \big( s_1 \otimes \partial_{\overline{z}} - \partial_{\overline{z}}\otimes s_2\big) \, (f\otimes g)\ ,$$
where $\partial_z = \partial/\partial z$, or again as a product of  commuting determinants
$$\Big(f, g \Big)_2 =  \left|\begin{matrix} \partial_z \otimes \id  & r_1 \\ \id \otimes \partial_z & r_2 \end{matrix} \right|   \left|\begin{matrix} s_1  &  \partial_{\overline{z}} \otimes \id   \\ s_2 &  \id \otimes \partial_{\overline{z}}  \end{matrix} \right|  (f\otimes g) \ . $$
\end{example}

Interesting operators of order two in the ring $\Or \otimes \Or$ therefore include: the Laplace operators $\Delta\otimes \id$ and $\id \otimes \Delta$, the Rankin-Cohen bracket $[f,g]_2$ and its conjugate, and a  symmetric product $(f,g)_2$.  All this is part of the general study of differential operators on $\mathcal{M}$ and $\mathcal{M}^!$, which we shall not pursue any further here.

\section{Modular forms and equivariant sections} \label{sectEquivariant}
In this section, all tensor products are over $\Q$. 

\subsection{Reminders on representations of $\SL_2$.}
For all $n \geq 0$ define
$$V_{2n} = \bigoplus_{r+s = 2n} X^r Y^s \Q\ ,$$
equipped with the right action of $\SL_2(\Z)$ given by 
$$(X,Y)\big|_{\gamma} =   (aX+ bY, cX+ dY)$$
for $\gamma  $ of the form  $(\ref{gammaact})$. There is an isomorphism  of $\SL_2$-representations
$$V_{2m} \otimes V_{2n} \cong V_{2m+2n} \oplus V_{2m+2n-2} \oplus \ldots \oplus V_{2|m-n|}\ .$$
We shall use the following choice of  $\SL_2$-equivariant  projector 
\begin{equation} 
\delta^k  : V_{2m} \otimes V_{2n} \To V_{2m+2n- 2k}
\end{equation} 
by   setting 
$$\delta^k =   m  \circ \Big( {\partial \over \partial X} \otimes {\partial \over \partial Y} - {\partial \over \partial Y}\otimes {\partial \over \partial X}  \Big)^k  $$
where $m: \Q[X,Y] \otimes  \Q[X,Y]  \rightarrow  \Q[X,Y]$ is the multiplication map. For an equivalent formulation and further properties, see \cite{MMV}.

\subsection{A characterisation of functions in $\mathcal{M}$} See also \cite{Zemel}, Proposition 2.1.

\begin{prop}  \label{propmodularformsfromsections} Let $f: \HH \rightarrow V_{2n}\otimes \C$ be real analytic. Then it can be written in two equivalent manners: either in the form 
\begin{equation} \label{fsection1} f= \sum_{r+s=2n}  f^{r,s}(z) X^r Y^s
\end{equation} 
for some real analytic functions  $f^{r,s} : \HH \rightarrow \C$, or in the form 
\begin{equation} \label{fsection2} f= \sum_{r+s=2n} f_{r,s} (z)(X-zY)^r (X- \overline{z} Y)^s \ , 
\end{equation}
where  $ (z- \overline{z})^{2n} f_{r,s} : \HH \rightarrow \C$ are real analytic. 
The function $f$ is equivariant:
\begin{equation} \label{fequivariant}  f(\gamma(z)) \big|_{\gamma} = f(z) \qquad \hbox{ for all }  \quad \gamma \in \SL_2(\Z)
\end{equation} 
if and only if the coefficients $f_{r,s}$ are modular $(\ref{fbimod})$ of weight $(r,s)$. 
Suppose now that the coefficients $(\ref{fsection1})$ of $f$ admit expansions  of the form
$$f^{r,s} \in \C[[q, \overline{q}]][z, \overline{z} ]\ .$$
Then  $f$ is equivariant if and only if $f_{r,s}  \in P^{-r-s} \mathcal{M}_{r,s}\ .$
\end{prop} 
\begin{proof} First observe that the  inclusion 
$$ \Z[z,\overline{z}] [(X-zY), (X-\overline{z} Y)]  \To \Z[z,\overline{z}] [ X, Y]$$
becomes an isomorphism after inverting  $z- \overline{z}$. Indeed, the inverse is given  by 
\begin{eqnarray}  \label{XYinverse}
X &  \mapsto  &   {z \over z- \overline{z}}  (X-\overline{z}Y) -   {\overline{z} \over z- \overline{z}}  (X- zY)     \\
 Y &  \mapsto  &   {1 \over z- \overline{z}}  (X-\overline{z}Y) -   {1 \over z- \overline{z}}  (X- zY) \ .\nonumber 
 \end{eqnarray} 
This proves that the expansions $(\ref{fsection1})$ and $(\ref{fsection2})$ are equivalent.  The identity
$$(X- \gamma(z) Y) \big|_{\gamma} =  {\det(\gamma) \over (c z +d)}  (X- zY)$$
 implies that $(X-zY)^r (X-\overline{z} Y)^s$  transforms, under the simultaneous action of $\SL_2(\Z)$  on the argument $z$ in the usual manner and on the right of $V_{2n}$, like a modular function of weights $(-r,-s)$.  The   coefficient of  $(X-zY)^r (X-\overline{z} Y)^s$ for $f$   equivariant is therefore modular of weights $(r,s)$. 
    For the last statement, the assumption on the Fourier expansions of $f^{r,s}$ implies that 
 the coefficients $(z-\overline{z})^{r+s} f_{r,s}$ admit expansions in  the ring $\C[z, \overline{z}][[q, \overline{q}]]$ by $(\ref{XYinverse})$. By  lemma \ref{lemlogqexpansion}, the $f_{r,s}$  have expansions of the form  $(\ref{qexp})$.
\end{proof} 

We construct equivariant functions $f$ from iterated integrals. These only involve non-negative powers of $\log q$.  Their coefficients $f_{ij}$  will have  poles in $\LL$ of degree at most the total weight, and their modular weights 
will naturally  be located in the first quadrant $r,s\geq 0$.

\subsection{Vector-valued differential equations} The operators  $\partial, \overline{\partial}$ of definition \ref{partialdefn} admit the following interpretation. 
\begin{prop} \label{propdF=A}  Let $F,A,B: \mathcal{H} \rightarrow V_{2n}\otimes \C$ be real analytic. Then the equation 
\begin{equation} \label{dF=A} {\partial F \over \partial z} = {2 \pi i  \over 2}  A(z)
\end{equation} 
is equivalent to the system of equations  for all $r+s = 2n$, and $r,s \geq 0$:
\begin{eqnarray}
 \partial F_{2n,0}  & = &   \LL  A_{2n,0}   \\
 \partial  F_{r,s} - (r+1) F_{r+1, s-1}  &= & \LL  A_{r,s} \qquad \hbox{ if } \quad s\geq 1 \ .  \nonumber 
\end{eqnarray} 
In a similar manner, 
\begin{equation} \label{dbarF=B} {\partial F \over \partial \overline{z}} = {2 \pi i  \over 2}  B(z)
\end{equation} 
is equivalent to the following system of equations:
\begin{eqnarray}
 \overline{\partial}  F_{0,2n}  & = &   \LL  B_{0,2n}   \\
 \overline{\partial} F_{r,s} - (s+1) F_{r-1, s+1}  &= & \LL  B_{r,s} \qquad \hbox{ if } \quad r\geq 1 \ .  \nonumber 
\end{eqnarray} 

\end{prop} 

\begin{proof} Differentiate the expression $(\ref{fsection2})$ to obtain 
$${\partial F \over \partial z} = \sum_{r+s=2n} \Big( { \partial  F_{r,s} \over \partial z} (X-z Y)^r (X- \overline{z} Y)^s - rY F_{r,s} (X-z Y)^{r-1} (X- \overline{z} Y)^s\Big)  \  .$$
On replacing $Y$ using the second line of $(\ref{XYinverse})$, the right-hand side becomes
$$ \sum_{r+s=2n} \Big( \Big( {\partial  F_{r,s} \over \partial z} + {r F_{r,s} \over z- \overline{z}}\Big)    (X-z Y)^r (X- \overline{z} Y)^s - { r F_{r,s} \over z-\overline{z}}  (X-z Y)^{r-1} (X- \overline{z} Y)^{s+1}\Big)\ . $$
Multiplying through by $z- \overline{z}$, collecting terms and using  definition $\ref{partialdefn}$, we see that the equation $(\ref{dF=A})$ is equivalent  to the system of equations
$$ \partial_r F_{r,s} - (r+1) F_{r+1, s-1} = i \pi  (z- \overline{z}) A_{r,s} $$
for $ 1 \leq s \leq 2n$, and in the case $s=0$, 
$ \partial_{2n} F_{2n,0}  =  i \pi  (z- \overline{z})  A_{2n,0}$.
Conclude using $(\ref{LLdef})$. The second set of equations can be deduced by conjugation. 
\end{proof}
The commutation relation $[\partial , \overline{\partial}]=\sfh$ of proposition $\ref{propsl2}$  is equivalent to 
$${\partial^2 F \over \partial z \partial \overline{z} } =  {\partial^2  F \over \partial  \overline{z} \partial z } \ . $$

\begin{lem} \label{lemapplydeltak}  Suppose that $A: \HH \rightarrow V_{2n}$ and set
$$F= {\delta^k \over (k!)^2} \Big(  (X-zY)^{2m} \otimes A \Big)\ . $$
Then  $F: \HH \rightarrow V_{2m+2n-2k}$ vanishes if $k>2n $ or $k>2m$, but otherwise satisfies
\begin{equation} \label{Frsequationinlemmadeltak}
F_{ r,s }  =    (z- \overline{z})^k \,     \binom{2m}{k}\binom{s+k}{k} A_{r-2m+k, s+k}   
  \end{equation}
where we set $A_{p,q}=0$ for $p<0$ or $q<0$. Therefore  $F_{r,s}$ vanishes   if $r<2m-k$, or equivalently, $s+k>2n$. 
\end{lem} 
\begin{proof} By direct application of the definition of $\delta^k$, we find that 
\begin{multline} {\delta^k \over (k!)^2}  \Big( (X- zY)^{2m} \otimes A(X,Y) \Big)  \nonumber  =    \\ (z-\overline{z})^k  \sum_{r, s, r+2m\geq k}   \binom{2m}{k}\binom{s}{k} A_{r,s} (X-zY)^{r+2m-k} (X- \overline{z}Y)^{s-k}  \end{multline}
where $$A(X,Y) = \sum_{r+s = 2n}  A_{r,s} (X-zY)^r (X-\overline{z} Y)^s\ . $$ 
Equation  $(\ref{Frsequationinlemmadeltak})$ follows on replacing $(r,s)$ with $(r-2m+k,s+k)$. 
\end{proof} 

Combining the  lemma with proposition \ref{propdF=A}, we find that if 
$${\partial F \over \partial z}  = {(\pi i)^{k+1}} {\delta^k \over (k!)^2} \Big(   f_{2m+2}(z) (X-zY)^{2m}\otimes A(X,Y)\Big)$$
then 
\begin{equation}  \label{finalequation} \partial F_{2n,0} =  \binom{2m}{k}\LL^{k+1}  f_{2m+2}(z)  A_{2n-2m+k,k} (z)\ . \end{equation}

\subsection{Example: real analytic Eisenstein series} 

Let us write, for $w>0$ even:
$$\mathcal{E}_w(z)= \sum_{r+s=w}  \mathcal{E}_{r,s} (X-zY)^r (X-\overline{z}Y)^s \ . $$
It is equivariant for $\SL_2(\Z)$.  Consider  also the equivariant 1-form 
\begin{equation}\label{Eunderlinedefn} \underline{E}_{w+2}(z) = 2 \pi i \, \GE_{w+2}(z) (X- zY)^{w} dz \ . 
\end{equation} 
 Then by proposition (\ref{dF=A}) the systems of equations $(\ref{realanalholequation})$ and $(\ref{realanalantiholequation})$ are equivalent to the following differential equation:  
 \begin{eqnarray} d \mathcal{E}  &  =  & { 1 \over 2}  \Big( \underline{E}_{w+2}(z)   + \overline{\underline{E}_{w+2}(z)}  \Big)  \nonumber \\ 
&  = &   \mathrm{Re} \, \big(  \underline{E}_{w+2}(z)  \big)  \ . \nonumber
\end{eqnarray} 
The real analytic Eisenstein series  of \S\ref{sectRAEisenstein} are the real parts of  primitives of  vector-valued holomorphic Eisenstein series \cite{MMV}, \S9.2.2.  This motivates a general construction of modular forms in $\mathcal{M}$ via equivariant iterated integrals of modular forms.

\section{Modular forms from equivariant iterated integrals}
The  main idea behind our construction of functions in $\mathcal{M}$ is a modification of the theory of single-valued periods as we presently explain in some  simple  examples.
\subsection{Single-valued functions}
  The logarithm 
\begin{equation} \label{logzasint} \log z = \int_1^z {dt \over t}
\end{equation} 
is a multi-valued analytic function on $\C^{\times}$.  This means that its pull-back to  the universal covering space $\C$ of $\C^{\times}$ based at $1$ is an analytic function.  Indeed, via the covering map $\exp: \C \rightarrow \C^{\times}$, it  simply corresponds to  the function $z$ on $\C$.  
The fundamental group of $\C^{\times}$ at the  point $1$  is isomorphic to $\Z$, and is generated by a simple loop around the origin. Analytic continuation  around this loop creates a discontinuity
\begin{equation} \label{logdiscont} \log z  \mapsto \log z + 2  \pi  i \ .
\end{equation}
Since the monodromy period $2  \pi i $ is purely imaginary, the multivaluedness of the logarithm can be eliminated by taking its real part:
$$\log |z| =    \mathrm{Re}\,  \big(\! \log z\big) \ .$$
This is the `single-valued' version of the logarithm. It  is a well-defined function on $\C^{\times}$, invariant under the left action of $\Z= \pi_1(\C^{\times}, 1)$.  The dilogarithm
$$\Li_2(z) = \sum_{k \geq 1} {z^k \over k^2}$$
defined for $|z|<1$ and analytically continued to a multivalued function on $\C^{\times} \backslash \{0,1\}$,
 satisifes the equation $d \Li_2(z) = - \log(1-z) {dz\over z}$, and has a single valued version :
 $$ D(z)= \mathrm{Im} \, \Big( \Li_2(z) + \log |z| \log(1-z) \Big)  $$
called the Bloch-Wigner function. In the following sections we construct modular analogues of the functions $\log|z|$  and $D(z)$.

There is a  general way to associate single-valued functions to any period integrals \cite{NotesMot} \S4, \S8.3 generalising $(\ref{logzasint})$.
The latter can depend on parameters, or even be constant.  
A variant of this construction, applied to iterated integrals of modular forms, yields a class of functions in $\mathcal{M}$. 
This follows from proposition \ref{propmodularformsfromsections} since a  real analytic  section 
$f: \HH \rightarrow V_n\otimes \C$ is equivariant if and only if the coefficients $f_{r,s}$  in the expansion
$f = \sum_{r,s} f_{r,s} (X-z Y)^{r} (X-\overline{z} Y)^{s} $ are modular of weights $(r,s)$. The equivariance $$  f(\gamma(z))\big|_{\gamma} = f(z)  \qquad \hbox{ for all } \gamma \in \SL_2(\Z)   $$
can be interpreted as  single-valuedness of the vector-valued function 
$f(z)$ on the orbifold quotient of $\HH$ by the action of $\SL_2(\Z)$.

\subsection{Notation} For $f$ a holomorphic modular form of weight $n$, let us denote by 
$$ \underline{f}(z) = 2 i \pi  f(z) (X-zY)^{n-2} dz\ .$$
It is invariant under the action of  $\SL_2(\Z)$ on $z$ and $X,Y$. We shall  write 
\begin{equation}\nonumber
S= 
\left(
\begin{array}{cc}
  0   & -1  \\
   1  &   0 
\end{array}
\right)\quad, \quad  T= \left(
\begin{array}{cc}
  1   &  1  \\
   0  &   1 
\end{array}
\right)
\ .
\end{equation}

\subsection{The modular function $\mathrm{Im}(z)$} 
Denote by 
$$F(\tau) = \int_{\tau}^{i \infty}  \underline{1} = 2 \pi i  \int_{\tau}^{i \infty} (X-z Y)^{-2} dz =  {  2\pi i \over Y(X-zY)}\ .$$
We obtain in this manner a $\Q(X,Y)$-valued cocycle 
$$F(\gamma(\tau))\big|_{\gamma} - F(\tau) = C_{\gamma}$$
where $C: \SL_2(\Z) \rightarrow \Q(X,Y)$ is the function 
$$C_{\gamma} =  - {c \, 2\pi i \over Y (Xc +Yd ) } \quad \hbox{ where } \quad    \gamma = \begin{pmatrix} a & b\\c & d 
  \end{pmatrix}  \ .$$
This cocycle is cuspidal (vanishes for $\gamma=T$). 
Since $C_{\gamma}$ is  imaginary, the real  part $\mathrm{Re}\, F(\tau)$ is modular equivariant. Indeed we have
$$\mathrm{Re}\, F(\tau)   =   {\LL \over (X-\tau Y) (X - \overline{\tau} Y)} $$
which is  $\SL_2(\Z)$-invariant since  $\LL $ is  modular of weights $(-1,-1)$.

\subsection{Primitives of holomorphic modular forms} Now we construct, or fail to construct, equivariant versions of classical Eichler integrals in the same vein. 

\subsubsection{Cusp forms} Let $f\in S_{2n}$ be a cusp form with rational Fourier coefficients. 
Let \begin{equation} \label{Ftaudef} F(\tau) =  \int_{\tau}^{i \infty} \underline{f}(z)  =   2 \pi i \int_{\tau}^{i \infty} f(z) (X-zY)^{2n-2} dz\ .   
\end{equation}
It satisfies, by invariance of $\underline{f}$, the following monodromy  equation  
\begin{equation} \label{Fdiscont} F(\gamma(\tau))\big|_{\gamma} \ = \  F(\tau) + C_{\gamma}\ ,
\end{equation}
for all $\gamma \in \SL_2(\Z)$ and $\tau \in \HH$.  It is the analogue of $(\ref{logdiscont})$. It follows from $(\ref{Fdiscont})$ that the function  $C: \SL_2(\Z) \rightarrow \C[X,Y] $ is a cocycle for $\SL_2(\Z)$, and indeed that 
$$C_S  =  C_S^+   +  i\, C_S^{-} $$ 
where $C_S^{+},C_S^- \in \R[X,Y]$ are the even and odd real period polynomials of $f$.    By the Eichler-Shimura theorem,  the classes of $C^+$ and $ C^-$ are independent in group cohomology $H^1(\SL_2(\Z), \C[X,Y])$, and  so there is no way, by taking real and imaginary parts, that we can kill the right-hand side of $(\ref{Fdiscont})$ to obtain a single-valued function.
Therefore  a single iterated integral, or primitive,  of a cusp form yields nothing new. 
Indeed, by proposition \ref{propdF=A}, such a function, if it existed, would  provide a solution to the equation $\partial F = f$,  in $\mathcal{M}$,  which would contradict  \ref{cornocuspprim}. This obstruction can be circumvented by introducing poles; cusp forms do have primitives in   $\mathcal{M}^!$ (see \S \ref{sectMercusp}).

\subsubsection{Eisenstein series}  \label{sectEisasint}  If $f=\G_{2k+2}$  
the corresponding integral  $(\ref{Ftaudef})$ diverges, but can be regularised in the manner of \cite{MMV} \S4, yielding a primitive 
$$F(\tau) =   \int_{\tau}^{\tone_{\infty}} \underline{E_{2k+2}}(z) $$
 which satisfies $(\ref{Fdiscont})$.  Here, $\tone_{\infty}$  denotes the unit tangential basepoint  at the cusp.  The associated Eisenstein cocycle $C$  satisfies (\cite{MMV}, \S7): 
$$C_{\gamma}  =   {(2k)! \over 2 } {  \zeta(2k+1) \over   (2\pi i)^{2k}}  \,  Y^{2k}\Big|_{\gamma - 1}   +(2 \pi i ) e^0_{2k+2}(\gamma)$$
where $e^0_{2k+2}(\gamma) \in \Q[X,Y]$.  Now if we consider the real part of $F(\tau)$, it satisfies the analogue of $(\ref{Fdiscont})$ with $C$ replaced with $\mathrm{Re}\, C$. The  key point is that the real part of $C_{\gamma}$ only involves the first term in the previous equation,  which is a coboundary for $\SL_2(\Z)$. Therefore,   the function $\mathrm{Re}\, F(\tau)$   can be  modified  in the following manner to define 
  a vector-valued real-analytic function 
$$  \mathcal{E}_{2k}(X,Y)(\tau) \ =      \mathrm{Re}\, F(\tau)    \ +  \    {(2k)!  \over 2 } {  \zeta(2k+1) \over   (2\pi i)^{2k}}  Y^{2k}$$  
which is now invariant under the action of $\SL_2(\Z)$. 
 This   function can  be rewritten
 $$ \mathcal{E}_{2k}(X,Y)(z) = \sum_{r+s= 2k}  \mathcal{E}_{r,s} (z)(X-zY)^r (X-\overline{z} Y)^s$$
 where the functions $\mathcal{E}_{r,s}:\HH \rightarrow \C$  are  modular, and lie in $\mathcal{M}_{r,s}$.
As the notation suggests, the coefficients are precisely the real analytic functions $\mathcal{E}_{r,s}$ defined  in \S\ref{sectRAEisenstein}.
They are the analogues of the single-valued functions $\log|z|$ on $\C^{\times}$.

\begin{rem}
 The systematic use of tangential basepoints to regularise period integrals associated to Eisenstein series clarifies and simplifies many  constructions in the literature.   Since this was only recently introduced \cite{MMV} \S4, we provide some commentary:
 
 \begin{itemize}
 
 \item The period polynomial of the Eisenstein series is equivalent to formulae which must have been known to Ramanujan and are given by \cite{MMV}, \S9:
 \begin{eqnarray} \   e_{2k}^0(S)   &= &  {(2k-2)! \over 2} \, \sum_{i=1}^{k-1}    {\Be_{2i} \over (2i)!}{\Be_{2k-2i}  \over (2k-2i)!}  X^{2i-1} Y^{2k-2i-1} \  ,  \nonumber   \\
  e_{2k}^0(T)   &= &  {(2k-2)! \over 2}  {\Be_{2k} \over (2k)!} \Big(  { (X+Y)^{2k-1} - X^{2k-1}   \over Y}\Big)\ .  \nonumber 
  \end{eqnarray}  
 However, I could not find  this precise formulation elsewhere. The literature  tends to focus on period polynomials (value of a cocycle on $S$) which only determine the cocycle in the cuspidal case. Zagier's approach is to  introduce poles in $X,Y$ to force the Eisenstein cocycle to be cuspidal.
 
 \item It is often stated that $X^{2n} - Y^{2n}$ is the  period polynomial of an Eisenstein series, but is in fact  the value of the cuspidal coboundary cocycle at $S$, and vanishes in cohomology. It is, however, non-zero in \emph{relative} cohomology and is dual to the Eisenstein cocycle under the Petersson inner product (which pairs cocycles and compactly supported cocycles). This is discussed in \cite{MMV} \S9.
 
 \item The `extra' relation satisfied by period polynomials of cusp forms \cite{KohnenZagier},  expresses the orthogonality of the cocycle of a cusp form to the   Eisenstein cocycle with respect to the Haberlund-Petersson inner product.
 
 \end{itemize}

 \end{rem}

\section{Equivariant double iterated integrals} \label{sect: EquivDoubleEis}
We now define equivariant versions of  double Eisenstein integrals, which are modular analogues of the Bloch-Wigner function $D(z)$.

\subsection{Double Eisenstein integrals}
Recall that 
$$\mathcal{E}_{2n}: \HH \To V_{2n} \otimes \C$$
is the modular-invariant real analytic function which satisfies
\begin{equation} \label{dE2nisE} d \,\mathcal{E}_{2n} =   \mathrm{Re} \, \underline{E}_{2n+2}\ . 
\end{equation}
For every $a,b\geq 2$ consider the family of one forms
\begin{eqnarray} 
\mathcal{D}_{2a,2b}   \ :  \  \HH &\To &    \big( V_{2a-2} \otimes V_{2b-2}  \big) \otimes \big( \C\, dz + \C \, d\overline{z} \big) \nonumber \\ 
\mathcal{D}_{2a,2b} & = &     \underline{E}_{2a} \otimes \mathcal{E}_{2b-2} + \mathcal{E}_{2a-2} \otimes  \overline{\underline{E}_{2b}}  \nonumber
\end{eqnarray}
They  are modular invariant: for all $\gamma \in \SL_2(\Z)$ we have
 $$\mathcal{D}_{2a,2b} (\gamma z)\big|_{\gamma} = \mathcal{D}_{2a,2b}(z)\ .$$
 \begin{lem} The family of forms $\mathcal{D}_{2a,2b}$ are closed:
 $$d \mathcal{D}_{2a,2b} =0 $$
\end{lem}

\begin{proof}  By $(\ref{dE2nisE})$ and writing   $ \overline{dz} \wedge dz=- dz \wedge \overline{dz} $, we find that
$$   d \,   \mathcal{D}_{2a,2b} =  -   \underline{E}_{2a} \otimes \overline{\underline{E}}_{2b}    +\underline{E}_{2a} \otimes \overline{\underline{E}}_{2b} = 0 \ .   \   $$ 
\end{proof}
 By the previous lemma, it makes sense to consider the indefinite integral 
\begin{equation} K_{2a,2b}(z)  \ =  \frac{1}{2} \int_{z}^{\tone_{\infty}} \mathcal{D}_{2a,2b}(z)\ , 
\end{equation}
since the integrand is closed, and the integral only depends on the homotopy class of the chosen path. 
This function can be written in terms of real and imaginary parts of products  of iterated integrals of Eisenstein series. Indeed, we have
$$ K_{2a,2b}(z) \equiv    \frac{1}{2i}  \, \mathrm{Im}\,\Big( \int_{z}^{\tone_{\infty}}  \underline{E}_{2a} \underline{E}_{2b} \Big)   - \frac{1}{2} \mathrm{Re}  \,\Big( \int_{z}^{\tone_{\infty}}  \underline{E}_{2a} \Big)\times    \int_{z}^{\tone_{\infty}}  \overline{\underline{E}_{2b}  } $$
modulo iterated integrals of length one (we integrate from left to right).  

The real analytic function 
$$K_{2a,2b}(z): \HH \To V_{2a-2} \otimes V_{2b-2} \otimes \C$$  satisfies  the pair of differential equations
\begin{eqnarray}
{\partial \over \partial z} K_{2a,2b}  & = & {2 \pi i   \over 2}   \GE_{2a} (z)(X-zY)^{2a-2}  \otimes \mathcal{E}_{2b-2}(z)   \nonumber \\
{\partial \over \partial \overline{z}} K_{2a,2b}  & = &   {2\pi   i \over 2}    \mathcal{E}_{2a-2}(z) \otimes  \overline{\GE_{2b}(z)} (X-\overline{z}Y)^{2b-2}   \nonumber
\end{eqnarray} 
Recall the normalised projection 
$$ (  \pi i  )^k {\delta^k  \over (k!)^2} :    V_{2a-2} \otimes V_{2b-2} \otimes \C \To V_{2a+2b-4-2k} \otimes \C$$
\begin{defn} For any $a,b\geq 1$ and $0 \leq k \leq \min \{2a, 2b\}$, define 
$$K_{2a+2,2b+2}^{(k)}(z) = ( \pi i  )^k {\delta^k  \over (k!)^2} K_{2a+2,2b+2}(z) \ .$$
\end{defn}
By equation $(\ref{finalequation})$, its  lowest-weight components  satisfy
\begin{equation} \label{lowestweightforKs} \partial \,  \big(K_{2a+2,2b+2}^{(k)}\big)_{2a+2b-2k,0}   \quad  =  \quad  \binom{2a}{k} \LL^{k+1} \GE_{2a+2} \mathcal{E}_{2b-k,k}   
\end{equation}

\subsection{Equivariant versions of double Eisenstein integrals }
Since $\mathcal{D}_{2a,2b}$, and hence $d\, K^{(k)}_{2a,2b}$ is modular equivariant, it follows that 
$$C_{\gamma} \quad = \quad K^{(k)}_{2a,2b}  (\gamma \tau) \big|_{\gamma} - K^{(k)}_{2a,2b}(\tau) $$
is constant (does not depend on $\tau$), and defines a cocycle 
$$C_{\gamma }  \quad \in \quad Z^1(\SL_2(\Z), V_{2a+2b-4-2k}\otimes \C)\ .$$
By the Eichler-Shimura theorem, any such  cocycle can be expressed as a linear combination of cocycles of
cusp forms or their complex conjugates, Eisenstein series, and a coboundary $c |_{\gamma - \id}$ for some $c\in V_{2a+2b-4-2k}\otimes \C$. Define a modified function
\begin{equation} \label{Mkdef} M^{(k)}_{2a,2b} = K^{(k)}_{2a,2b}  - c - \frac{1}{2}   \int_{z}^{\tone_{\infty}} \big(\underline{f} +  \overline{\underline{g}} \big) 
\end{equation}
where $f$ is a holomorphic modular form, and $g$ a cusp form, both of weight $2a+2b-2-2k$,  which is modular equivariant:
$$M^{(k)}_{2a,2b} (\gamma \tau) \big|_{\gamma} =   M^{(k)}_{2a,2b} (\tau) \ .$$
This equation uniquely determines the functions $f, g, c$ and  $M^{(k)}_{2a,2b}$, except in the case when $2a+2b-4-2k=0$ since we can add an arbitrary constant $c\in \C$. Extracting the coefficients of $M^{(k)}_{2a,2b}$ via $(\ref{fsection2})$ yields a class of functions in $\mathcal{M}$.

\begin{thm}  \label{thmEquivDoubleEis} Let $a, b \geq 1$ and $0\leq k \leq \min \{2a,2b\} $, and  set  $w= a+b-k$.  There exists a family of elements $F_{r,s} \in \MI_2 \cap \mathcal{M}_{r,s}$ of total modular weight  $2w=r+s$ where $r,s\geq 0$,  which satisfy the equations 
\begin{eqnarray} \label{Frsequations}
\partial F_{r,s}  - (r+1) F_{r+1,s-1} & = & \binom{2a}{k} \binom{k+s}{k} \LL^{k+1} \GE_{2a+2} \mathcal{E}_{2b-k-s, k+s} \qquad  \hbox{ if } s\geq 1 \nonumber\\ 
\partial F_{2w,0}   & = & \binom{2a}{k} \LL^{k+1} \GE_{2a+2} \mathcal{E}_{2b-k, k}   \quad + \quad  \LL\, f  
\end{eqnarray}
where $f$ is the unique cusp form of weight $2w+2$ satisfying 
\begin{equation}\label{phtogetfinthm}   p^h   \Big( \binom{2a}{k} \LL^{k+1} \GE_{2a+2} \mathcal{E}_{2b-k, k}   +  \LL\, f  \Big) = 0\ 
\end{equation}
and $\mathcal{E}_{m,n}$ is understood to be zero if either  of $m, n$ are  negative. 
\end{thm}
 \begin{proof}  The function  $M^{(k)}_{2a+2,2b+2} : \HH \rightarrow V_{2w}\otimes \C$ is equivariant by definition.  Let $M_{r,s}$ denote its modular components obtained from proposition \ref{propmodularformsfromsections}. From the definition $(\ref{Mkdef})$, and the differential equation for $K^{(k)}_{2a+2,2b+2}$, we can apply proposition $\ref{propdF=A}$ to deduce the equations satisfied by 
  $M_{r,s}$.  The first line of $(\ref{Frsequations})$ follows since the correction terms in $(\ref{Mkdef})$
  only affect the case $s=0$. For $s=0$,  $(\ref{Mkdef})$ and $(\ref{lowestweightforKs})$ imply that 
  $$\partial M_{2w, 0} =   \binom{2a}{k} \LL^{k+1} \GE_{2a+2} \mathcal{E}_{2b-k, k}   \quad + \quad  \LL\, f  \ .$$
 By modifying $M^{(k)}_{2a+2,2b+2}$ by  a suitable multiple of $\mathcal{E}_{2w+2} $, we can assume that $f$ is
 a  cusp form. It is uniquely determined by  theorem $\ref{thmpartialsorthogtocusp}$, which gives equation $(\ref{phtogetfinthm})$. 
 
The quantities $\overline{\partial} M_{r,s}$ can be computed from $(\ref{Mkdef})$, and satisfy:\begin{eqnarray} \label{Frsconjequations}
\overline{\partial} M_{r,s}  - (s+1) M_{r-1,s+1} & = &  \binom{2b}{k} \binom{k+s}{k} \LL^{k+1} \GE_{2b+2} \mathcal{E}_{2a-k-s, k+s}  \nonumber\\ 
\overline{\partial} M_{0,2w}   & = &  \binom{2b}{k} \LL^{k+1} \GE_{2b+2} \mathcal{E}_{2a-k, k}   \quad + \quad  \LL\, g  
\end{eqnarray}
This proves that the $M_{r,s}$ lie in $\MI_{2}$.  
The cusp form $g$ is uniquely determined from the antiholomorphic projection  $p^{a}( \overline{\partial} M_{0,2w})=0$. 
  \end{proof}
  
  \begin{rem}
 The function $K^{(k)}_{2a,2b}$ is related to the function $I^{(k)}_{2a,2b}$ defined in \cite{MMV}, and its holomorphic projection is possibly related to the double Eisenstein series of \cite{DC}.    \end{rem} 
  
  The system of functions satisfying both $(\ref{Frsequations})$ and $(\ref{Frsconjequations})$ are unique up to possible addition of a multiple of $\LL^{-r}$ for $F_{r,r}$. 
  We can show  that these functions are  linearly independent for distinct values of $a,b$ and $k$.

\subsection{Example} \label{ExamplesE4E4} Since there are no cusp forms in weights $\leq 10$, it follows that the functions defined above, for 
$2w = 2a+2b-2k \leq 8  $
only involve iterated integrals of Eisenstein series.
 The simplest possible example  is the case 
$a=1,b=1$ and $k=0,1,2$. 
The  equivariant iterated integral $M^{(k)}_{4,4}$  of $\GE_4$ and $\GE_4$ solves the equation 
$$ {\partial F^{(k)}  \over \partial z }  ={ (i \pi)^{k+1}}  {\delta^{(k)} \over (k!)^2}   \big( \underline{\GE_4} \otimes \mathcal{E}_2 \big)$$
for $k =0, 1,2$, corresponding to the three components of 
$$\delta^{0}\oplus \delta^{1} \oplus \delta^{2} : V_{2} \otimes V_2 \overset{\sim}{\To} V_4 \oplus V_2 \oplus V_0 \ . $$ By proposition $\ref{propdF=A}$, this equation is equivalent to the following three families of equations, which we spell out for  concreteness. In the case $k=0$, we have 
\begin{eqnarray}
\partial F^{(0)}_{0,4} -  F^{(0)}_{1,3} & =&  0  \\
\partial F^{(0)}_{1,3} - 2 F^{(0)}_{2,2} & =&  0  \nonumber \\
\partial F^{(0)}_{2,2} -  3F^{(0)}_{3,1} & =&  \LL \GE_4 \mathcal{E}_{0,2}   \nonumber \\
\partial F^{(0)}_{3,1} -  4 F^{(0)}_{4,0} & =&   \LL \GE_4 \mathcal{E}_{1,1}  \nonumber \\
\partial F^{(0)}_{4,0}    & =&  \LL \GE_4 \mathcal{E}_{2,0} \ .   \nonumber 
\end{eqnarray} 
In the case $k=1$, we have
\begin{eqnarray}
\partial F^{(1)}_{0,2} -  F^{(1)}_{1,1} & =&  0   \\
\partial F^{(1)}_{1,1} - 2  F^{(1)}_{2,0} & =& 4  \LL^2 \GE_4 \mathcal{E}_{0,2}   \nonumber \\
\partial F^{(1)}_{2,0}   & =& 2  \LL^2 \GE_4 \mathcal{E}_{1,1}  \ .   \nonumber \end{eqnarray} 
Finally, in the case $k=2$ we have the single equation 
\begin{equation}
\qquad \qquad \partial F^{(2)}_{0,0}  \quad =  \quad   \LL^3 \GE_4 \mathcal{E}_{0,2}  \ . 
\end{equation}
Since there are no cusp forms in weight 4, the statement of  theorem  \ref{thmpartialsorthogtocusp} is vacuous and there is no  obstruction 
 to the  existence of a solution of these equations. The constant terms can be computed from the double integral of the Eisenstein series $E_4$ and $E_4$. The latter involves at most $\zeta(3), \zeta(5)$ and $\zeta(3)^2$.

The shuffle product formula for iterated integrals implies that 
$$
\begin{array}{ccccccc}
 F_{0,4}^{(0)}  &=  &\textstyle{1 \over 2} \mathcal{E}_{0,2}^2 &   &   F_{4,0}^{(0)} & =& \textstyle{1 \over 2} \mathcal{E}_{2,0}^2 \\
 F_{1,3}^{(0)} &=&  \mathcal{E}_{0,2} \mathcal{E}_{1,1}  &   &  F_{3,1}^{(0)}& =&  \mathcal{E}_{2,0} \mathcal{E}_{1,1}  \\
 &&   F^{(0)}_{2,2} & =& \mathcal{E}_{2,0} \mathcal{E}_{0,2} + \textstyle{1\over 2} \mathcal{E}_{1,1}^2   &&
\end{array}
$$
are linear combinations of real analytic Eisenstein series as is 
$$F^{(2)}_{0,0}  =  \LL^2( \mathcal{E}_{2,0} \mathcal{E}_{0,2}  - \textstyle{1\over 4} \mathcal{E}_{1,1}^2  )$$
 but this is not the case for the  functions $F^{(1)}$, which are new.    The above identities can be verified from their differential equations. Observe that the functions $\mathcal{E}_{2,0} \mathcal{E}_{0,2}$ and $\mathcal{E}_{1,1}^2$  can be expressed as linear combinations of $F^{(0)}_{2,2}$ and $\LL^{-2} F^{(2)}_{0,0}$. 
  By theorem $\ref{thmEquivDoubleEis}$  
\begin{eqnarray}
\overline{\partial} F^{(1)}_{0,2}  & =&  2  \LL^2 \overline{\GE_4} \mathcal{E}_{1,1}  \nonumber    \\
\overline{\partial} F^{(1)}_{1,1} - 2  F^{(1)}_{0,2} & =& 4  \LL^2  \overline{\GE_4} \mathcal{E}_{2,0}   \nonumber \\
\overline{\partial} F^{(1)}_{2,0} -  F^{(1)}_{1,1}  & =& 0  \ ,   \nonumber \end{eqnarray} 
and  $\overline{\partial} F^{(2)}_{0,0}  \quad =  \quad   \LL^3 \overline{\GE_4} \mathcal{E}_{2,0}$. 
 Now, from the above equations and the definition \ref{LaplaceDef} of the Laplacian, we deduce the equations:
$$
\begin{array}{ccccccc}
 (\Delta + 2)  F_{0,2}^{(1)} & = &  - 4 \LL^2   \overline{\GE_4} \mathcal{E}_{2,0}    & \qquad & (\Delta + 4)  F_{0,4}^{(0)} & = &  - \LL \overline{\GE_4}   \mathcal{E}_{1,1}   \\
(\Delta + 2)  F_{1,1}^{(1)} & = &  - 4 \LL^3  \GE_4 \overline{\GE_4}    &&   (\Delta + 4)  F_{1,3}^{(0)} & = &  -2  \LL \overline{\GE_4}   \mathcal{E}_{2,0}     \\ 
(\Delta + 2)  F_{2,0}^{(1)} & = &  - 4 \LL^2   \GE_4 \mathcal{E}_{0,2} && (\Delta + 4)  F_{2,2}^{(0)} & = &  -  \LL^2 \GE_4  \overline{\GE_4}      \\ 
&&&& (\Delta + 4)  F_{3,1}^{(0)} & = &  - 2  \LL \GE_4  \mathcal{E}_{0,2}  \\ 
\Delta F^{(2)}_{0,0}  &  = &  -  \LL^4 \, \GE_4 \overline{\GE_4} && (\Delta + 4)  F_{4,0}^{(0)} & = &  -   \LL \GE_4  \mathcal{E}_{1,1}   
\end{array}
$$
In fact, we note that
$$(\Delta+2) \mathcal{E}_{1,1}^2 = - 8\,  \mathcal{E}_{0,2} \mathcal{E}_{2,0} \qquad \hbox{ and } \qquad \Delta\,  \mathcal{E}_{2,0} \mathcal{E}_{0,2}  = -  \mathcal{E}_{1,1}^2 -\LL^2 \GE_4 \overline{\GE_4}\ .$$
We have therefore generated three modular-invariant functions 
$$ \LL^2 \mathcal{E}_{2,0} \mathcal{E}_{0,2} \quad , \quad \LL^2 \mathcal{E}_{1,1}^2 \quad , \quad \LL F^{(1)}_{1,1} \qquad \in \qquad\mathcal{M}_{0,0} $$
out of which one can construct solutions to inhomogeneous Laplace equations:
$$(\Delta + 2) \big( \LL F^{(1)}_{1,1} - 4 \LL^2 \mathcal{E}_{2,0} \mathcal{E}_{0,2}\big) =  4 \LL^2 \mathcal{E}_{1,1}^2\ , $$
By comparing with $(\ref{introC211evalue})$, this suggests that the modular graph function $C_{2,1,1}(z)$ can be expressed in terms of  the functions $\LL F^{(1)}_{1,1} - 4 \LL^2 \mathcal{E}_{2,0} \mathcal{E}_{0,2}$, $\LL^3 \mathcal{E}_{3,3}$ and a  constant.

\begin{rem} The coefficients in the expansion (\ref{qexp}) of these functions  are easily determined from the formulae for  the action (\ref{partialactformula}) of $\partial, \overline{\partial}$ and the above differential equations, up to the sole exception of  a constant term. This constant term can be determined by analytic continuation since it is the residue of a certain zeta function canonically associated to elements in $\mathcal{M}_{r,s}$.  This will be discussed elsewhere. 
\end{rem} 

\subsection{Orthogonality conditions}
 We now wish to consider the problem of finding linear combinations of equivariant iterated integrals which only involve Eisenstein series, i.e., in which all  integrals of cusp forms cancel  out.  This is equivalent to finding linear combinations of the $M_{2a,2b}^{(k)}$ which are orthogonal to all cusp forms under the Petersson inner product. Since this problem is discussed 
 in \cite{MMV}, \S22 in an essentially equivalent form, we illustrate with a simple example.

\begin{example}
Let $k=0$, and let $\Delta$ denote the Hecke  normalised cusp form of weight $12$. Consider the four lowest weight functions   $F_{10,0}^{2a+2,2b+2} \in  \mathcal{M}_{10,0}$  
for $a+b=5$ and $1\leq a,b$ which are described in theorem  \ref{thmEquivDoubleEis}, and satisfy the equations
$$\begin{array}{rclclcrclcl}
 \partial F^{2,8} &  =&   \LL \GE_4 \mathcal{E}_{8,0} &-& \alpha^{2,8} \LL \Delta  &\quad   , \quad &\partial F^{8,2}  &= &  \LL \GE_{10} \mathcal{E}_{2,0} &- &\alpha^{10,4} \LL \Delta    \\
  \partial F^{4,6}  &= &  \LL \GE_6 \mathcal{E}_{6,0}& -& \alpha^{4,6} \LL \Delta& ,  &  \partial F^{6,4}  &= &  \LL \GE_8 \mathcal{E}_{4,0}& -& \alpha^{6,4} \LL \Delta     
\end{array}
$$
where $\alpha^{2a,2b}$ are determined by the Petersson inner product:
$$  \alpha^{2a,2b} \langle \Delta, \Delta \rangle = \langle \GE_{2a+2} \mathcal{E}_{2b,0}, \Delta\rangle \ . $$
The right-hand side can be computed by the Rankin-Selberg method \cite{MMV} \S9, and implies that 
$\alpha^{2a, 2b} $ is proportional, by some explicit factors, to   $L (\Delta, 2a+1) L(\Delta, 12)$.
On the other hand, it is well-known that the quantities $L(\Delta, k)$ for $1 \leq k \leq 11$ satisfy the period polynomial relations, and we deduce that the quantity
$$ X=  9 \big(F^{2,8}-F^{8,2}\big)  + 14 \big(F^{4,6}-F^{6,4})     $$
has the property that all terms involving $\Delta$ drop out of $\partial X$. 
One can show, furthermore,  that  $X$ is dual to the relations between double zeta values in weight $12$.
\end{example}

Viewed in this manner, it might seem hopeless to find iterated integrals of Eisenstein series of higher lengths which are equivariant. Already in length three,  the Rankin-Selberg method can no longer be applied in any obvious manner  to  
 find the necessary linear combinations of triple Eisenstein integrals. Fortunately, using the theory of the motivic fundamental group of the Tate curve, we can find an infinite class, and conjecturally all, solutions  to this problem. This is summarised below.

\section{A space of  equivariant Eisenstein integrals} \label{sectMIE}

Recall that $E$ is the graded $\Q$ vector space generated by Eisenstein series $(\ref{Espacedefn})$. 
Let $\ZZ^{\sv}$ denote the ring of single-valued multiple zeta values. 

\begin{thm} There exists a 
 space $\MI^E \subset \mathcal{M}$ with the following properties:
\begin{enumerate} \setlength\itemsep{0.02in}
\item It is the $\ZZ^{\sv}$-vector space generated by certain (computable) linear combinations of real and imaginary parts of regularised iterated integrals of Eisenstein series. 
\item The space $\MI^E[\LL^{\pm}]$ is stable under  multiplication and complex conjugation.
\item It carries an even filtration (conjecturally a grading)  by $M$-degree, where $\LL$ has $M$-degree $2$, and the $\mathcal{E}_{r,s}$ have $M$-degree $2$. 
It is also filtered by the length (number of iterated integrals), which we denote by $\MI^E_k \subset \MI^E$. 
\item The subspace  of elements of  $\MI^E$ of total  modular weight $w$  and $M$-degree $\leq m$ is finite-dimensional for every $m,w$. 
\item Every element of $\MI^E$ admits an expansion in the ring
$$\ZZ^{\sv} [[q, \overline{q}]][\LL^{\pm}]\ ,$$
i.e., its coefficients are single-valued multiple zeta values.  An element of total modular  weight $w$ has poles in $\LL$ of order at most $w$.  An element of $M$-degree $2m$ has terms in $\LL^k$ for $k \leq m$. 
\item The space $\MI^E$ has the following differential structure:
\begin{eqnarray} 
\partial  \big( \MI_k^E\big)    &\subset &    \MI^E_{k}  \quad +  \quad E[\LL]  \times \MI_{k-1}^E    \nonumber \\
\overline{\partial}  \big( \MI_k^E \big)  &\subset &    \MI^E_{k} \quad + \quad   \overline{E}[\LL]  \times \MI_{k-1}^E    \ . \nonumber 
\end{eqnarray} 
The operators $\partial, \overline{\partial}$ respect the $M$-grading, i.e., $\deg_M \partial= \deg_M \overline{\partial}=0$, 
where the generators $\GE_{2n+2}$ of $E$ are placed in $M$-degree $0$.  

\item Every element $F \in \MI_k^E$ of total modular weight $w$ satisfies an inhomogeneous Laplace equation of the  form:
$$ (\Delta + w )\, F  \quad \in \quad  (E+ \overline{E})[\LL] \times \MI^E_{k-1} + E \overline{E} [\LL]  \times \MI^E_{k-2} \ . $$  
\end{enumerate}
\end{thm}
Explicitly, we have
$\MI^E_0 =  \quad \ZZ^{\sv} $, and 
$$\MI^E_1 \quad  =   \quad  \ZZ^{\sv} \oplus \bigoplus_{r,s\geq0} \mathcal{E}_{r,s}\ZZ^{\sv}   $$
In length $2$, we can show that $\MI_2^E$ is generated by the coefficients of linear combinations of $M^{(k)}_{2a,2b}$ which do not involve any cusp forms, and in particular the example of \S\ref{ExamplesE4E4}.   The previous theorem can be compared with \S\ref{SectModGraphProperties}.

\begin{rem}
A more precise statement about the  Laplace equation (7) can be derived from the differential equations (6). In fact, the differential equations with respect to $\partial, \overline{\partial}$ are the more fundamental structure. This simplicity is  obscured when looking only at the Laplace operator. 
Recently,  a generalisation of modular graph functions called modular graph forms were introduced in \cite{Graph3}. They define functions in $\mathcal{M}$ of more general modular weights $(r,s)$, and, up to scaling by $\LL^{\pm}$,   are closed under the action of $\partial, \overline{\partial}$. 
  It suggests that one should try to find systems of  differential equations, with respect to $\partial, \overline{\partial}$, satisfied by modular graph forms using partial fraction identities (see \cite{Graph3}, (2.30)), and match their solutions with elements in $\MI^E[\LL^{\pm}]$. 

We briefly explain how the previous theorem relates to a recent observation in \cite{Graph5} for modular graph functions.  
Suppose that  $f \in \MI^E$ of modular weights $(w,w)$. Then $\LL^{w} f$ is modular invariant, and by $(7)$ and repeated application of $(\ref{DeltaLL})$   it satisfies an inhomogenous Laplace eigenvalue equation with eigenvalue  
$$- (2w + 2w-2 + 2w-4 + \ldots + 2 + 0 ) = - 2 \binom{w}{2}  =   -w(w-1)\ .$$
It was observed in \cite{Graph5} that dihedral modular graph functions satisfy an inhomogeneous Laplace equation with eigenvalue $-s(s-1)$, where $s$ is a positive integer, and  the same statement was proved in \cite{Graph6} for two-loop modular graphs functions using the representation theory of $\mathrm{SO}(2,1)$.

The $M$-filtration can be made more precise. If $F \in \MI^E$ of $M$-degree $\leq 2 m$ then the coefficient of $\LL^{-k}$ in the constant part $F^0$ of $F$ is a single-valued multiple zeta value of weight  $ \leq k +m.$ If one assumes (for example, by replacing multiple zeta values with their motivic versions) that multiple zeta values are graded, rather than filtered, by weight, then 
this filtration would also be a grading. 
For example,  the elements $\mathcal{E}_{r,s}$ have constant parts
$$\mathcal{E}^0_{r,s}   \quad \in \quad    \LL\,  \Q +   \LL^{-r-s} \, \zeta(r+s+1) \Q\ .$$
The (MZV)-weight of $\zeta(r+s+1)$ is (conjecturally) $  r+s+1$, and the weight of a rational number is $0$. 
This is entirely consistent with $\deg_M \mathcal{E}_{r,s}=2$. 
\end{rem} 

This theorem and  further  properties of $\MI^E$ will be proved in the sequel.

\section{Meromorphic primitives of cusp forms} \label{sectMercusp}

We revisit the problem of finding primitives of cusp forms. If we allow poles at the cusp, then we can indeed construct  modular equivariant versions of cusp forms \cite{Mock}.

\subsection{Weakly analytic variant of $\mathcal{M}$}
Let $\mathcal{M}_{r,s}^{!}$ denote the vector space of functions $f:\HH \rightarrow \C$ which are real analytic modular of weights $(r,s)\in \Z^2$ admitting an expansion of the form
$$f (q)= \sum_{k=-N}^N \LL^k \, \Big(\! \sum_{m,n\geq -M} a^{(k)}_{m,n}   q^m \overline{q}^n\Big) $$
for some integers $M, N \in \N$, i.e., with  poles in $q, \overline{q}$ at $0$. Let 
$$\mathcal{M}^{!} =  \bigoplus_{r,s} \mathcal{M}_{r,s}^{!}\ .$$
It is a bigraded algebra and satisfies   $ \mathcal{M}^! = \mathcal{M}[\Delta(z)^{-1}, \overline{\Delta(z)}^{-1}]$ where $\Delta(z)$ denotes the Hecke normalised cusp form of weight $12$.  
This ring of functions satisfies similar properties to  $\mathcal{M}$, and  is equipped with operators $\partial, \overline{\partial}, \Delta$ as defined earlier.

\begin{defn} Define a space of modular iterated integrals $\MI^! \subset \mathcal{M}$ as follows. Let $\MI^!_{-1}=0$ and let $\MI^!_k \subset \mathcal{M}$ be the largest subspace which is contained in the positive quadrant (modular weights $(r,s)$ with $r,s\geq 0$), such that 
\begin{eqnarray}  \partial \MI^!_k   & \subset &  \MI^!_k  +  M^![\Lef]\times   \MI^!_{k-1} \nonumber \\
 \overline{\partial} \MI^!_k  & \subset & \MI^!_k  +  \overline{M^!}[\Lef] \times \MI^!_{k-1}\ . \nonumber
 \end{eqnarray}
 \end{defn} 
We now give some examples of elements  in $\MI^!_k$ for $k\leq 2$. 

\subsection{Primitives of cusp forms} The following theorem is proved in \cite{Mock}

 \begin{thm} For every cusp form $f\in S_n$, there exists a canonical family of functions 
  $\mathcal{H}(f)_{r,s}$ for all $r,s \geq 0$, with $r+s=n$ satisfying
 \begin{eqnarray} 
 \partial \, \mathcal{H}(f)_{n,0}   &= & \LL  f  \nonumber  \\ 
 \partial \, \mathcal{H}(f)_{r,s} & = & ( r+1 )   \mathcal{H}(f)_{r+1, s-1}     \qquad  \hbox{ for all  }1\leq s\leq w \  \nonumber 
\end{eqnarray} 
and 
 \begin{eqnarray} 
\overline{\partial} \, \mathcal{H}(f)_{0,n}   &= & \LL \,\overline{\s(f)}  \nonumber  \\ 
 \overline{\partial} \, \mathcal{H}(f)_{r,s} & = &  (s+1) \mathcal{H}(f)_{r-1, s+1}      \qquad  \hbox{ for all  }1\leq r\leq w \  \nonumber 
\end{eqnarray} 
where $\s(f) \in S_n^!$ is a weakly holomorphic modular form canonically associated to $f$.  They are eigenfunctions of the Laplacian with eigenvalue $-n$.
   \end{thm} 
   If we write 
   $$\mathcal{H}(f) = \sum_{r+s=n}  \mathcal{H}(f)_{r,s} (X- zY)^r (X- \overline{z} Y)^s$$
   then the  system of  equations above are equivalent to 
   \begin{equation} \label{dHf} 
    d \mathcal{H}(f) = \pi i  f(z)(X-zY)^n dz +      \pi i \,  \overline{\s(f)(z)}(X-\overline{z} Y)^n d\overline{z} 
   \end{equation} 
   
    In \cite{MMV} \S18,   these formulae were generalised to  all higher order iterated integrals.  
   We show in \cite{Mock} that $\MI^!_0 = \C [\Lef^{-1}]$, and prove:
   \begin{thm}  $\MI^!_1$ is the free $\C [\Lef^{-1}]$-module generated by the $\mathcal{H}(f)_{r,s}$\ .
   \end{thm} 
    
    \subsection{New elements in $\mathcal{M}_{r,s}$}
    By multiplying by a suitable power of $\Delta(z) \overline{\Delta(z)}$ to clear the poles at the cusp,  we obtain elements in $\mathcal{M}$. For every cusp form $f\in S_n$,  
    $$\overline{\Delta(z)}^N  \mathcal{H}(f)_{r,s}  \in \mathcal{M}$$
    for sufficiently large $N$ (in fact, $N= \dim S_n$ will do). In particular, 
    $$\overline{\Delta(z)}  \mathcal{H}(\Delta(z))_{r,s} \in \mathcal{M}_{r,s+12} \ .$$
    This provides further evidence that the space of modular forms $\mathcal{M}$ contains potentially interesting elements. 
\subsection{Double integrals} Having defined the weakly holomorphic modular primitives of cusp forms, we can use them to construct equivariant double integrals of an Eisenstein series  and a cusp form, or two cusp forms. The definition is along very similar lines to \S\ref{sect: EquivDoubleEis}: consider the indefinite integrals of  the one-forms:
$$   \underline{f} \otimes \overline{\mathcal{H}(g)} +  \mathcal{H}(f)\otimes \overline{\underline{g}}  \qquad \hbox{ or } \qquad
\underline{f} \otimes \mathcal{E}_n  + \mathcal{H}(f) \otimes\overline{\underline{E}}_{n+2}
$$
They are closed by (\ref{dHf}), and so their indefinite integrals are well-defined (homotopy invariant).  The general strategy is  always the same: let 
$$\Omega = \sum_{r+s=n}   \omega_{r,s} (X-zY)^r (X-\overline{z}Y)^s$$
with $d \Omega=0$, and $\omega_{r,s} \in \mathcal{M}^!_{r,s}$ (which implies that $\Omega(\gamma z) \big|_{\gamma} = \Omega(z)$ for all $\gamma \in \SL_2(\Z)$). Consider the indefinite integral
$$\mathcal{F}(z) = \int_{z}^{z_0} \Omega $$
where $z_0 \in \HH$ is any point. Then 
$$\gamma \mapsto \mathcal{F}(\gamma z) \big|_{\gamma} - \mathcal{F}(z)  \quad \in \quad Z^1(\SL_2(\Z); V_n\otimes \C)$$
is a cocycle.   By  the Eichler-Shimura theorem, we can 
add primitives of holomorphic modular forms, anti-holomorphic cusp forms, and constants to $\mathcal{F}$ to  make this cocycle vanish (this is a generalisation of the proof of  lemma \ref{lemFmodular}).
The resulting function is therefore modular equivariant.   Extracting the coefficients in the manner of proposition \ref{propmodularformsfromsections}, we obtain non-trivial functions in  $\mathcal{MI}^!_2$.

As above, by multiplying by sufficiently large powers of $\Delta (z)\overline{\Delta}(z)$, we can clear poles in the denominators to obtain yet more elements in $\mathcal{M}$, and so on. 

\bibliographystyle{plain}
\bibliography{main}

\end{document}